\documentclass{amsart}

\usepackage{amsthm}
\usepackage{graphicx}
\usepackage{enumerate}
\usepackage{amsmath,amssymb,amscd,amsfonts}
\usepackage{mathtools}
\usepackage{color}

\allowdisplaybreaks[4]

\newtheorem{theorem}{Theorem}[section]
\newtheorem{lemma}[theorem]{Lemma}
\newtheorem{proposition}[theorem]{Proposition}
\newtheorem{corollary}[theorem]{Corollary}
\newtheorem{sublemma}[theorem]{Sublemma}

\theoremstyle{definition}

\newtheorem{example}[theorem]{Example}

\theoremstyle{remark}
\newtheorem{remark}[theorem]{Remark}

\numberwithin{equation}{section}

\newcommand{\N}{\mathbb{N}}

\newcommand{\R}{\mathbb{R}}
\newcommand{\C}{\mathbb{C}}

\newcommand{\cA}{\mathcal{A}}

\newcommand{\cL}{\mathcal{L}}
\newcommand{\cN}{\mathcal{N}}
\newcommand{\cP}{\mathcal{P}}

\newcommand{\1}{\mathbf{1}}

\newcommand{\al}{\alpha}
\newcommand{\be}{\beta}
\newcommand{\la}{\lambda}
\newcommand{\vep}{\varepsilon}
\newcommand{\wh}{\widehat}

\newcommand{\var}{\operatorname{var}}
\newcommand{\BV}{\widehat{BV}}

\begin{document}

\title[Non-leading eigenvalues ]{Non-leading eigenvalues of the Perron-Frobenius operators for beta-maps}


\author{Shintaro Suzuki}
\address{Department of Mathematics, 
Tokyo Gakugei University, 
4-1-1 Nukuikita-machi 
Koganei-shi, 
Tokyo 184-8501,
Japan}
\email{shin05@u-gakugei.ac.jp}


\begin{abstract}
We consider the Perron-Frobenius operator defined on the space of functions of bounded variation for the beta-map $\tau_\beta(x)=\beta x$ (mod $1$), for $\beta\in(1,\infty)$, and investigate its isolated eigenvalues 
except $1$, called non-leading eigenvalues in this paper. We show that the set of $\beta$'s such that the corresponding Perron-Frobenius operator has at least one non-leading eigenvalue is open and dense in $(1,\infty)$. Furthermore, we establish the H\"older continuity of each non-leading eigenvalue as a function of $\beta$ and show in particular that it is continuous but non-differentiable, whose analogue was conjectured by Flatto et.al. in \cite{Fl-La-Po}. In addition, for an eigenfunctional of the Perron-Frobenius operator corresponding to an isolated eigenvalue, we give an explicit formula for the value of the functional applied to the indicator function of every interval. 
As its application, we provide three results related to non-leading eigenvalues, one of which states that an eigenfunctional corresponding to a non-leading eigenvalue can not be expressed by any complex measure on the interval, which is contrast to the case of the leading eigenvalue $1$.

\end{abstract}
\maketitle
\begin{section}{Introduction}


Piecewise monotone maps of the interval with some hyperbolicity have been investigated as simple models of the so-called chaotic dynamical systems. For such a map one of the key tools for investigating its ergodic properties is the corresponding Perron-Frobenius operator, whose spectrum may reflect information about the measurable dynamics associated to its absolutely continuous (with respect to the Lebesgue measure) invariant probability measures (abbreviated to a.c.i.p.m.). In the case that a map is piecewise $C^2$ and piecewise expanding, the Perron-Frobenius operator defined on a Banach space of suitable functions (e.g., functions of bounded variation) is a linear bounded operator, whose spectral radius is $1$ and this value is actually an eigenvalue, which yields the existence of an a.c.i.p.m. (see \cite{La-Yo}). Furthermore, the Perron-Frobenius operator is to be quasi-compact, i.e., there is a real number $0 < r < 1$ such that a spectral value whose modulus greater than $r$ is an isolated eigenvalue with finite multiplicity ( 
the infimum of such numbers is called the essential spectral radius). If $1$ is a simple eigenvalue then there exists a unique a.c.i.p.m. and the measurable dynamics associated to the a.c.i.p.m is ergodic (see \cite{Li-Yo}). Moreover, if the modulus of any other eigenvalue is less than $1$, a simple eigenvalue $1$ is the leading eigenvalue (in the sense of its modulus) and we have  exponential decay of correlations for suitable observables (see e.g. \cite{Ba}, \cite{Bo-Go}, \cite{Ho-Ke0}).  In this case we say that the Perron-Frobenius operator has a spectral gap property and it can be applied to show several limit theorems via perturbation theory, for which the analyticity of the leading eigenvalue of the perturb Perron-Frobenius operator plays an important role (see e.g. \cite{Bo-Go},\cite{Ho-Ke0}). 

Contrast to the case of the leading eigenvalue, the properties of the other isolated eigenvalues greater than the essential spectral radius, called non-leading eigenvalues in this paper, may still remain unclear, including their existence. 
In \cite{De-Fr-Se} Dellnitz  et. al. constructed a family of piecewise linear expanding Markov maps on the interval each of whose Perron-Frobenius on the space of functions of bounded variation has a positive non-leading eigenvalue. In \cite{Bu-Ki-Li}, Butterley et.al. investigated an asymptotic expansion of a correlation function for suitable observable in case of piecewise linear expanding Markov maps or piecewise full branched expanding maps. In particular they gave an example of a piecewise linear map such that its Perron-Frobenius operator on a space of suitable functions has an isolated eigenvalue whose algebraic multiplicity and index both are $2$ (see Section 2 for notation). 

In this paper, we investigate some properties of non-leading eigenvalues of the Perron-Frobenius operators for beta-maps defined on the space of functions of bounded variation, 
each of which has a spectral gap property. The beta-maps are defined by $\tau_\be(x)=\be x$ (mod $1$), $x\in[0,1]$, $\be\in(1,\infty)$, and have been well-studied after \cite{Re} in the context of ergodic theory and number theory since each map generates an expansion of a real number in $[0,1]$ whose base is $\be$, called the greedy expansion (see e.g. \cite{Bl}, \cite{Da-Kr}, \cite{Pa1}).  
The advantage of the map is that several quantities related to the Perron-Frobenius operator can be given in concrete terms, such as a normalized eigenfunction corresponding to the leading eigenvalue $1$, i.e., the density function of a unique a.c.i.p.m (\cite{Ge}, \cite{Pa1}, \cite{Re}). In fact, our main results are derived from the two explicit formulas which we describe in detail in the following. 

The first explicit formula is for the dynamical zeta function of a beta-map, derived from the result for a beta-shift by Takahashi \cite{Ta}. As we see in Section 2, it is regarded as a sort of Fredholm determinant of the Perron-Frobenius operator, which enables us to investigate isolated eigenvalues of the Perron-Frobenius operators via zeros of a certain analytic function on the open disk whose radius is $\be$ (Theorem \ref{phi}). 
Using this fact we show in Theorem \ref{main 1} that the set of $\be$'s such that the corresponding Perron-Frobenius operator has at least one non-leading eigenvalue is open and dense in $(1,\infty)$. Furthermore, we establish the H\"older continuity of each non-leading eigenvalue as a function of $\be$ (Theorem \ref{Main C}). In particular, we show that any non-leading eigenvalue is continuous but non-differentiable as a function $\be$, which shows in some sense a fractal property of non-leading eigenvalues. 
We note that an analogy of this result was conjectured by Flatto et.al. in \cite{Fl-La-Po}, whose statement is that the maximal modulus of any non-leading eigenvalue is nowhere differentiable as a function of $\be$.

The second explicit formula is for the value of an eigenfunctional for an isolated eigenvalue of the Perron-Frobenius operator applied to the indicator function of each interval (Proposition \ref{key2}). Using this formula we provide three applications in Section 4. 
As the first application, we characterize isolated eigenvalues of the Perron Frobenius operator from the differentiability of a certain function related to the greedy expansion, which is introduced as motivated by the explicit formula for an eigenfunctional (Theorem \ref{main a} (4) and (5)). As the second application, we show that an eigenfunctional corresponding to each non-leading eigenvalue can not be expressed by any complex measure (Theorem \ref{main b}), which is contrast to the case of the leading eigenvalue $1$ since an eigenfunctional corresponding to $1$ can be expressed by the Lebesgue measure. As the third application, we construct a family of step functions each of whose correlation function decays exponentially in the rate actually less than the modulus of the second eigenvalues (in the sense of their modulus) under the assumption that all the second eigenvalues are simple (Theorem \ref{main b2}). Concerning this result, in Appendix we provide three examples of a countable family of $\be$'s each of which satisfies the above assumption . 

This paper is organized as follows. 
In Section 2, we summarize basic notions used throughout this paper and show that the set of $\be$'s such that the corresponding Perron-Frobenius operator has at least one non-leading eigenvalue is open and dense in $(1,\infty)$. In Section 3, we establish the H\"older continuity of each non-leading eigenvalue as a function of $\be$. In Section 4, we show an explicit formula for an eigenfunctional of the Perron-Frobenius operator and provide its three applications. In Appendix, we construct three examples of countable sets of $\be$'s such
that the second eigenvalues of the Perron-Frobenius operator are simple.
\end{section}

\begin{section}{Preliminaries and known results}

In this section, we summarize some notions used throughout the paper. In subsection 2.1, we recall some basic properties of beta-maps. In subsection 2.2, we define the Perron-Frobenius operator for a beta-map and refer to its spectral decomposition. In subsection 2.3 we describe the analytic properties of the dynamical zeta function of a beta-map, which is regarded as a sort of `Fredholm determinant' of its Perron-Frobenius operator. 

\begin{subsection}{Basic properties of beta-maps}
For a real number $\be>1$, the beta-map $\tau_\be:[0,1]\to [0,1]$ is defined by
\[\tau_\be(x)=\be x-[\be x]\]
for $[0,1]$, where $[y]$ denotes the integer part of $y\geq0$. It is well-known that the map $\tau_\be$ gives the greedy expansion of $x\in[0,1]$ as follows. Since  
\[x=\frac{[\be x]}{\be}+\frac{\tau_\be(x)}{\be}\]
for $x\in[0,1]$, we have
\[\tau_\be^{n}(x)=\frac{[\be \tau_\beta^n(x)]}{\be}+\frac{\tau_\be^{n+1}(x)}{\be}\]
for $n\geq0$. Here we regard $\tau_\be^0(x)$ as $x$. By using the above equations inductively, we obtain
\[x=\sum_{n=1}^N\frac{[\be\tau_\be^{n-1}(x)]}{\be^n}+\frac{\tau^N_\be(x)}{\be^N}\]
for $N\geq1$. Set $a_n(\be,x)=[\be\tau_{\be}^{n-1}(x)]$ for $n\geq1$. Taking $N\to+\infty$ in the right side of the above equation provides the greedy expansion of $x$:  
\[x=\sum_{n=1}^\infty\frac{a_n(\be, x)}{\be^n}.\]
A real number $x\in[0,1]$ is said to be simple if there is a positive-integer $n_0\geq2$ such that $\tau_\be^{n_0-1}(x)\in\{1/\be,\dots,[\be]/\be\}$. In this case, we have $a_n(\be, x)=[\be\tau_{\be}^{n-1}(x)]=0$ for all $n\geq n_0+1$. 
We define the non-negative function $L=L_\be:[0,1]\to\N\cup\{0\}$ by $L(x)=n_0$ if $x$ is simple and $L(x)=\infty$ otherwise. 
$\be>1$ is said to be simple if $L(1)$ is finite.

In some situation, it is more useful to consider the quasi-greedy expansion of $x$ defined as follows. 
If $\be>1$ is simple, we set  
\[\{d_n(\be,1)\}_{n=1}^\infty=\overline{a_1(\be,1)\dots a_{L(1)-1}(\be,1)({a}_{L(1)}(\be,1)-1)}^{\infty},\]
where $\overline{b_1\dots b_k}^\infty$ denotes the infinite concatenation of a $k$-length word $b_1\dots b_k$ for $k\geq1$ with non-negative integers $b_1$,$\dots$, $b_k$.
If $\be>1$ is non-simple, we set $\{d_n(\be,1)\}_{n=1}^\infty=\{a_n(\be,1)\}_{n=1}^\infty$. By the definition of the sequence $\{d_n(\be,1)\}_{n=1}^\infty$ we have that $\displaystyle{1=\sum_{n=1}^\infty}d_{n}(\be,1)/\be^n$, which is called the quasi-greedy expansion of $1$.
For simple $x\in(0,1)$, we define
\[\{d_n(\be,x)\}_{n=1}^\infty=a_1(\be,x)\dots a_{L(x)-1}(\be,x)({a}_{L(x)}(\be,x)-1)d_1(\be,1)d_2(\be,1)\dots.\]
By setting $\{d_n(\be,x)\}_{n=1}^\infty=\{a_n(\be,x)\}_{n=1}^\infty$ for non-simple $x\in(0,1)$,
we also have that $\displaystyle{x=\sum_{n=1}^\infty}d_{n}(\be,x)/\be^n$,
which is called the quasi-greedy expansion of $x\in(0,1]$.


For $x\in[0,1]$ the coefficient sequence $\{a_n(\be,x)\}_{n=1}^\infty$ is a one-sided infinite sequence of non-negative integers in $\{0,1,\dots,[\be]\}^{\N}$. We equip the set $\{0,1,\dots,[\be]\}^{\N}$ with the lexicographic order $\prec$, which is defined by 
$\{p_n\}_{n=1}^\infty\prec\{q_n\}_{n=1}^\infty$ if there is a positive integer $m\geq1$ such that $p_n=q_n$ for $1\leq n<m$ and $p_m<q_m$ for $\{p_n\}_{n=1}^\infty, \{q_n\}_{n=1}^\infty\in \{0,1,\dots,[\be]\}^{\N}$. Note that this is a total order and its order topology 
coincides with the product topology derived from the discrete topology on $\{0,1,\dots,[\be]\}$. Since 
the function $x\mapsto\{a_n(\be,x)\}_{n=1}^\infty$ preserves the order relation,
the limit $\lim_{y\nearrow x}\{a_n(\be,y)\}_{n=1}^\infty$ exists. By the definition of the quasi-greedy expansion of $x$ we can see that $\{d_n(\be,x)\}_{n=1}^\infty=\lim_{y\nearrow x}\{a_n(\be,y)\}_{n=1}^\infty$. 
The next propositions are due to Parry \cite{Pa1} (see also \cite{Bl}), which are key tools for proving the main results in this paper. 

\begin{proposition}\label{Parry1}
Let $\be>1$ be a non-integer. For a sequence of integers $\{w_i\}_{i=1}^\infty\in\{0,1,\dots, [\be]\}^{\N}$ there is $x_0\in[0,1)$ such that $w_n=a_n(\be,x_0)$ for all $n\geq1$ if and only if 
\[\{w_{i+n}\}_{n=1}^\infty\prec\{d_n(\be,1)\}_{n=1}^\infty\]
for all $i\geq0$.
\end{proposition}

 
\begin{proposition}\label{Parry2}
Let $N$ be a positive integer. For a sequence of non-negative integers $\{w_i\}_{i=1}^\infty\in\{0,1,\dots, N\}^{\N}$, there is $\be>1$ such that $w_n=a_n(\be,1)$ for all $n\geq1$ if and only if 
\[\{w_{n+i}\}_{n=1}^\infty\prec \{w_n\}_{n=1}^\infty\]
for all $i\geq1$. 
\end{proposition}

\end{subsection}

\begin{subsection}{Perron-Frobenius operators}
This subsection is devoted to summarizing the basic properties of the Perron-Frobenius operator for a beta map and its spectral spectral decomposition (see e.g. \cite{Ba}, \cite{Ba-Ke}, \cite{Bo-Go}, \cite{La-Yo}, \cite{Su2}). 
For a function $f:[0,1]\to\C$, we define the total variation $\var(f)$ by
\begin{align*}
\var(f)=\sup\Biggl\{&\sum_{i=1}^n|f(x_i)-f(x_{i-1})|\ ;\ n\geq1, x_0=0, x_n=1 \\
&\text{ and }x_i\in[0,1] \text{ with } x_{i-1}<x_i \text{ for } 1\leq i\leq n\Biggr\}.
\end{align*}
Let us denote by $BV$ the set of functions of bounded variation, i.e., 
\[BV=\{f:[0,1]\to\C\ ;\ \var(f)<\infty\}.\]
We endow the space $BV$ with the norm 
\[||f||_{BV}=\var(f)+\sup_{x\in[0,1]}|f(x)|.\]
Then $(BV, ||\cdot||_{BV})$ is to be a Banach space (see e.g. \cite{Ba}).
The Perron-Frobenius operator $\cL_\be:BV\to BV$ is defined by 
\[\cL_\be f(x)=\frac{1}{\be}\sum_{x=\tau_\be(y)}f(y)\ \ (x\in[0,1])\] 
for $f\in BV$. Note that the Perron-Frobenius operator is linear bounded on $BV$ and satisfies the following dual property:
\[\int_{0}^1 \cL_\be f\cdot g dl=\int_0^1f\cdot g\circ\tau_\be dl\]
for $f\in BV$ and $g\in L^\infty(l)$, where $l$ denotes the Lebesgue measure on $[0,1]$ (see e.g. \cite{Ba}, \cite{Bo-Go}). 
Since the beta-map $\tau_\be$ is a piecewise linear and piecewise expanding map, the corresponding Perron-Frobenius operator is quasi-compact, i.e., a spectrum $\lambda\in\C$ with $|\la|>1/\be$ is an isolated eigenvalue with finite multiplicity (see e.g. \cite {Ba}, \cite{Ba-Ke}, \cite{Bo-Go}, \cite{Ho-Ke0}, \cite{Ho-Ke}). 


In terms of application, it is more useful for us to introduce a quotient space of $BV$, on which functions are identified if they coincide with each other except on some countable set. Let
\[\cN=\{f\in BV\ ;\exists\text{ a countable set }N \text{ s.t. } f(x)=0 \text{ for }x\in[0,1]\setminus N\}.\]
Note that the space $\cN$ is a closed linear subspace in $BV$.
Then we can define the quotient Banach space $(\widehat{BV}, ||\cdot||_{\widehat{BV}})$, where $\widehat{BV}=BV/\cN$ and $||f||_{\widehat{BV}}=\inf_{g\in\cN}||f+g||_{BV}$ for $f\in \BV$. 
As in the case of $BV$, the Perron-Frobenius operator defined on $\BV$ is linear bounded (see e.g. \cite{Ba}, \cite{Bo-Go}). By abuse of notation, we denote the Perron-Frobenius operator defined on $\BV$ by $\cL_\be$. 
Furthermore, its spectrum outside of the disk of radius $1/\be$ coincides with that of the Perron-Frobenius operator on $BV$ (see Proposition 3.4 in \cite{Ba}). 

Denote by $\widehat{BV}^*$ the set of complex-valued linear functionals on $\BV$. Let us define the dual operator of the Perron-Frobenius operator $\cL_\be^*:\BV^*\to\BV^*$ by 
\[\cL_\be^*(\nu(f))=\nu(\cL_\be f)\]
for $\nu\in\BV^*$ and $f\in\BV$. 

For an eigenvalue $\la\in\C$ of $\cL_\be$
the geometric multiplicity of $\la$ is the dimension of the eigenspace $\{f\in \BV\ ;\ (\la I-\cL_\be)f=0\}$, where $I$ denotes the identity map. The (algebraic) multiplicity of $\la$ is the dimension of the generalized eigenspace $\{f\in \BV\ ; \exists n\geq 1 \text{ s.t. }  (\la I-\cL_\be)^n f=0\}$. We call the value $\sup\{n\geq 1\ ;\ (\la I-\cL_\be)^n f=0\ \text{for some}\ f\in \BV\ \text{with}\ f\neq 0\}$ the index of $\la$. 

Remark that the map $\tau_\be$ has a unique invariant probability measure $\mu_\be$ absolutely continuous with respect to the Lebesgue measure, which yields that $1$ is a simple eigenvalue of $\cL_{\be}$ (see \cite{Re}). 
Furthermore, since the measurable dynamics $(\tau_\be, \mu_\be)$ is exact (see e.g. \cite{Bo}), we have that there is no isolated eigenvalue on the unit circle except the leading eigenvalue $1$ . 

For isolated eigenvalues of $\cL_\be$, we have the following:
\begin{proposition}[Theorem 3.3 in \cite{Su2}]\label{gm}
Let $\la\in\C$ be an isolated eigenvalue of $\cL_\be$ with multiplicity $M\geq1$. Then the geometric multiplicity of $\la$ is $1$ and the index of $\la$ is equal to $M$. 
\end{proposition}
Since $\cL_\be$ is quasi-compact on $\BV$, for $\theta\in(1/\be, 1)$ the spectral decomposition of $\cL_\be$ is given by
\begin{equation}\label{decomposition}
\cL_\be f=\sum_{i=1}^N \la_i h_i^{\bot} J_{i} \nu_{i}(f) +\cP\cL_{\be}
\end{equation} 
for $f\in\BV$, where $\la_i$ is an isolated eigenvalue with finite multiplicity $M_i$ with $|\la_i|>\theta$, $h_i$ is the vector of a basis $(h_{i,1},\dots, h_{i,M_i})$ for the corresponding generalized eigenspace, $J_i$ is a Jordan matrix composed by one Jordan block whose diagonals are $1$, and $\nu_i$ is the vector of eigenfunctionals $(\nu_{i,1},\dots,\nu_{i,M_i})$ satisfying $\nu_{i,j}(h_{k,l})=1$ if $i=k$ and $j=l$, and $\nu_{i,j}(h_{k,l})=0$ otherwise. 
Here the linear operator $\cP$ has the spectral radius less than or equal to $\theta$.

\end{subsection}

\begin{subsection}{Dynamical zeta functions}

In this subsection, we see that the dynamical zeta function of the map $\tau_\be$ can be extended to an analytic function given by the generating function of the coefficient sequence of the beta-expansion of $1$. 
The dynamical zeta function $\zeta_\be(z)$ of $\tau_\be$ is formally defined by 
\[\zeta_{\be}(z)=\text{exp}\Biggl(\sum_{n=1}^{\infty}\frac{z^{n}}{n}\sum_{x=\tau_\be^n x}\frac{1}{|(\tau_\be^n)'(x)|}\Biggr).\]
Note that the fact that $|(\tau_\be^n)'|=\be^n$ for $n\geq1$ yields
\begin{align}\label{zeta}
\zeta_\be(z)=\text{exp}\Biggl(\sum_{n=1}^{\infty}\frac{z^{n}}{n}\frac{\#\text{Fix}(\tau_{\be}^{n})}{\be^n}\Biggr),
\end{align}
where $\text{Fix}(\tau_\be^n)$ denotes the set of all fixed points of $\tau_\be^n$ and $\# A$ denotes the cardinality of $A$. 
By applying Theorem $2$ in \cite{Ba-Ke} to our setting, we have the following. 
\begin{theorem}[Application of Theorem $2$ in \cite{Ba-Ke}]\label{Ba}
1. The dynamical zeta function $\zeta_\be(z)$ converges absolutely in the open unit disk. In particular, it is analytic on the unit disk. 
Furthermore, it can be extended to $\{z\in\C\ ; |z|<\be\}$ as a meromorphic function. 

2. For $\la\in\C$ with $1/\be<|\la|\leq1$, $\la$ is an isolated eigenvalue of the Perron-Frobenius operator $\cL_\be$ with multiplicity $M$ if and only if $\la^{-1}$ is a pole of the dynamical zeta function $\hat{\zeta}_\tau(z)$ in $\{z\in\C\ ;\ |z|<\beta\}$ with multiplicity $M$.
\end{theorem}
This theorem enables us to investigate the properties of isolated eigenvalues of $\cL_\be$ via those of poles of $\zeta_\be(z)$.
In case of beta-maps, it is known that the analytic continuation of $\zeta_\be(z)$ can be given by the generating function of the coefficient sequence $\{a_n(\be,1)\}_{n=1}^\infty$. Let us define the power series $\phi_\be(z)$ by
\[\phi_\be(z)=\sum_{n=1}^\infty\frac{a_n(\be,1)}{\be^n}z^n.\]
It is easily seen that the convergence radius of $\phi_\be(z)$ is at least $\be$. The following theorem due to Flatto et.al. \cite{Fl-La-Po}, derived from the result by Takahashi \cite{Ta} for a beta-shift, shows an explicit formula for the analytic continuation of $\zeta_\be(z)$ (see also \cite{It-Ta} and \cite{Su1}). 

\begin{theorem}[Theorem 2.3 in \cite{Fl-La-Po}] 
\label{Ito}
For $z\in\C$ with $|z|<1$, we have 
\[\zeta_\be(z)=\frac{p_\be(z)}{1-\phi_\be(z)},\]
where $p_\be(z)=1-(z/\be)^{L(1)}$ if $\be$ is simple and $p_\be(z)=1$ if $\be$ is non-simple. 

\end{theorem}
 As a consequence of Theorem \ref{Ba} and \ref{Ito}, we have:
 
\begin{theorem}\label{phi}
A complex number $\la$ with $1/\be<|\la|\leq1$ is an isolated eigenvalue of $\cL_\be$ with multiplicity $M\geq1$ if and only if $\la^{-1}$ is a zero of $1-\phi_\be(z)$ with multiplicity $M\geq1$.
\end{theorem}

In some situations, it is more applicable to use the power series 
\[\Hat{\phi}_\be(z)=\sum_{n=1}^\infty\frac{d_n(\be,1)}{\be^n}z^n\]
instead of $\phi_\be(z)$, due to the left continuity of the map $\be\mapsto\{d_n(\be,1)\}_{n=1}^\infty$. Note that the convergence radius of $\Hat{\phi}_\be(z)$ is at least $\be$. The following proposition states that zeros of $1-\Hat{\phi}_\be(z)$ each of whose modulus is less than $\be$ coincides with those of $1-\phi_\be(z)$.

\begin{proposition}\label{a=d}
    For $z\in\C$ with $|z|<\be$, it is a zero of 
    $1-\Hat{\phi}_\be(z)$ if and only if  it is a zero of $1-\phi_\be(z)$. 
\end{proposition}

\begin{proof}
    If $\be>1$ is non-simple, the statement immediately follows from the fact that $\{d_n(\be,1)\}_{n=1}^\infty=\{a_n(\be,1)\}_{n=1}^\infty$, which yields $\Hat{\phi}_\be(z)=\phi_\be(z)$. 
    
    If $\be>1$ is simple, by the definition of the sequence $\{d_n(\be,1)\}_{n=1}^\infty$, we obtain
    \begin{equation*}
        \begin{aligned}
            \Hat{\phi}_\be(z)
            &=\sum_{n=1}^\infty\frac{d_n(\be,1)}{\be^n} \\
            &=\Biggl(1+\Bigl(\frac{z}{\be}\Bigr)^{L(1)}+\Bigl(\frac{z}{\be}\Bigr)^{2L(1)}+\cdots\Biggr)\Biggl(\sum_{n=1}^{L(1)}\frac{a_n(\be,1)}{\be^n}z^n-\Bigl(\frac{z}{\be}\Bigr)^{L(1)}\Biggr) \\
            &=\frac{1}{1-(z/\be)^{L(1)}}\Biggl(\phi(z)-\Bigl(\frac{z}{\be}\Bigr)^{L(1)}\Biggr)
            \end{aligned}
    \end{equation*}
    for $z\in\C$ with $|z|<\be$, which shows
    \begin{equation*}
        1-\Hat{\phi}_\be(z)=\frac{1}{1-(z/\be)^{L(1)}}(1-\phi_\be(z))
    \end{equation*}
    for $z\in\C$ with $|z|<\be$. This finishes the proof.
    \end{proof}

In the rest of this section, as an application of Theorem \ref{phi} and several results from \cite{Fl-La-Po}, we show that the set of $\be$'s such that the corresponding Perron-Frobenius operator $\cL_\be$ has at least one non-leading eigenvalue is open and dense in $(1,+\infty)$. 
Let us define the power series $\psi_\be(z)$ by 
\[\psi_\be(z)=1+\sum_{n=1}^\infty\frac{\tau_\be^n(1)}{\beta^n}z^n.\]
Since $\tau_{\be}^n(1)\in[0,1]$ for $n\geq1$ the convergence radius of $\phi_\be(z)$ is greater than or equal to $\be$.
We need the following lemma, whose 
proof is based on a basic calculation derived from the definition of the coefficient sequence of the greedy expansion of $1$.
\begin{lemma}[Proposition 4.1 (1) in \cite{Su1}]\label{lemma1-1}
For $z\in\C$ with $|z|<\be$ we have
\[1-\phi_{\be}(z)=(1-z)\psi_\be(z).\]
\end{lemma}

The proof of the following lemma is given in \cite{Fl-La-Po}. 
\begin{lemma}[Theorem 7.1 in \cite{Fl-La-Po}]\label{lemma1-2}
There is a dense set of $\be$'s such that the meromorphic function $\zeta_\be(\be z)$ has infinitely many poles in $(-1,0)$. 
\end{lemma}

By Theorem \ref{Ito}, Lemma \ref{lemma1-1} and Lemma \ref{lemma1-2}, we immediately have the following:

\begin{lemma}\label{lemma1-3}
There is a dense set of $\be$'s such that the analytic function $1-\phi_\be(z)$ has infinitely many zeros on $(-\be, 0)$.
\end{lemma}

The next lemma ensures the uniform convergence of the function $\psi_{\be}(z)$ on the open disk whose radius is less than $\be$.

\begin{lemma}[Lemma 3.3 in \cite{Fl-La-Po}]\label{lemma1-4} 
For a simple $\be_0\in(1,\infty)$, we have 
\[\lim_{\be\nearrow\be_0}\psi_\be(\be z)=\psi_{\be_0}(\be_0 z)/(1-(\be_0 z)^{L_{\be_0}(1)+1}), \ \ 
\lim_{\be\searrow\be_0}\psi_\be(\be z)=\psi_{\be_0}(\be_0 z)
\]
uniformly on any compact subset of $\{z\in\C\ ;\ |z|<\be_0\}$.
For a non-simple $\be_0\in(1,\infty)$, we have
\[\lim_{\be\to\be_0}\psi_\be(\be z)=\psi_{\be_0}(\be_0 z)\]
uniformly on any compact subset of $\{z\in\C\ ;\ |z|<\be_0\}$.
\end{lemma}

The main result of this section is the following:
\begin{theorem}\label{main 1}
For $N\geq1$ the set of $\be$'s such that $1-\phi_\be(z)$ has at least $N$ zeros in the unit open disk is open and dense in $(1,\infty)$. In particular, the set of $\be$'s such that $\cL_\be$ has at least one non-leading eigenvalue is open and dense in $(1,\infty)$.
\end{theorem}

\begin{proof}
For a positive integer $N\geq1$, let $\cA_N$ be the set of $\be$'s such that the meromorphic function $\zeta_\be(\be z)$ has at least $N$ poles in the open unit disk, that is, $1-\phi_\be(\be z)$ has at least $N$ zeros in the unit open disk. For $\be_0\in\cA_N$, take arbitrary $N$ zeros of $1-\phi_{\be_0}(\be_0 z)$ in the unit open disk. By Lemma \ref{lemma1-1}, Lemma \ref{lemma1-4} and Hurwitz's theorem, we can take an open interval $I$ including $\be_0$ such that for $\be\in I$ the analytic function $1-\phi_\be(\be z)$ has at least $N$ zeros in the unit open disk, which yields that the set $\cA_N$ is open. We note that the set of $\be$'s such that $1-\phi_\be(\be z)$ has infinitely many zeros on $(-1, 0)$ is dense by Lemma \ref{lemma1-3} and included in the set $\cA_N$, which shows that the set $\cA_N$ is dense. This finishes the proof.

\end{proof}


\begin{remark}
In Section 6 in \cite{Fl-La-Po}, Flatto et.al. constructed a countable set of numbers, called Enestr\"om numbers, such that the corresponding meromorphic function $\zeta_\be(\be z)$ has no pole in $\{z\in\C\ ;|z|<1\}$. This means that the Perron-Frobenius operator $\cL_\be$ has no non-leading eigenvalue when $\be$ is an Enestr\"om number, whose typical example is the golden ratio (see Theorem 6.1 in \cite{Fl-La-Po}). 
\end{remark}

\end{subsection}

\end{section}

\begin{section}{Non-differentiability of non-leading eigenvalues}

In this section, we investigate the H\"older continuity of each non-leading eigenvalue $\la_\be$ of the Perron-Frobenius operator $\cL_\be$ as a function of $\be$. In particular, we show in Theorem \ref{Main C} that the function $\be\mapsto\la_\be$ is non-differentiable, which shows in some sense a fractal property of non-leading eigenvalues.

Let $\be_0>1$ be a non-integer. 
By Theorem \ref{phi} and Lemma \ref{lemma1-1}
we know that $\la$ is a non-leading eigenvalue of $\cL_{\be_0}$ with multiplicity $M\geq1$ if and only if $\la^{-1}$ with $1<|\la^{-1}|<\be_0$ is a zero of $\psi_{\be_0}(z)$ with multiplicity $M\geq1$. 
Let $\la_{\be_0}$ be a non-leading eigenvalue of $\cL_{\be_0}$ with multiplicity $M\geq1$. 
 By Lemma \ref{lemma1-4} and Hurwitz's theorem, we know that there are open neighborhoods $U_{\be_0}$ of $\be_0$, $\Tilde{V}_{\be_0}$ of $\la_{\be_0}^{-1}$ and a set of complex numbers $\{\la_\be^{-1}\}_{\be\in U_{\be_0}}$ such that $\la_{\be}^{-1}$ is a zero of $\psi_{\be}(z)$ for $\be\in U_{\be_0}$ and the function $\be\in U_{\be_0}\mapsto\la_\be^{-1}\in \Tilde{V}_{\be_0}$ is continuous at $\be_0$. 
  Then the function $\be\in U_{\be_0}\mapsto \la_\be\in V_{\be_0}$ is also continuous at $\be_0$ and any $\la_\be$ in $V_{\be_0}$ is a non-leading eigenvalue of $\cL_\be$ for $\be\in U_{\be_0}$, where $V_{\be_0}=\{z^{-1}\in\C; z\in \Tilde{V}_{\be_0}\}$. Note that there may exist infinitely many functions as above in the case of $M\geq2$, while we take one from them and investigate its H\"older continuity throughout this section. 
  The main theorem of the section is the following:
\begin{theorem}\label{Main C}
Let $\be_0\in(1,\infty)$ be a non-integer and let $\la_{\be_0}$ be a non-leading eigenvalue of $\cL_{\be_0}$ with multiplicity $M\geq1$. Let $\be\in U_{\be_0}\mapsto \la_\be\in V_{\be_0}$ be a function as defined above. 

(1) Assume that $\displaystyle{\gamma_1:=\liminf_{n\to\infty}\sqrt[n]{\lim_{x\to1}\tau_{\be_0}^n(x)}}>0$ and set
\[\alpha_1=\frac{1}{M}\Biggl(\frac{\log|\be_0\la_{\be_0}|}{\log\be_0+\log\gamma_1^{-1}}\Biggr).\]
Then for any $\al<\al_1$ the function $\be\in U_{\be_0}\mapsto \la_\be\in V_{\be_0}$ is $\al$-H\"older continuous from left at $\be_0$.

(2) Assume that $\displaystyle{\gamma_2:=\liminf_{n\to\infty}\sqrt[n]{1-\tau_{\be_0}^n(1)}}>0$ and set
\[\alpha_2=\frac{1}{M}\Biggl(\frac{\log|\be_0\la_{\be_0}|}{\log\be_0+\log\gamma_2^{-1}}\Biggr).\]
Then for any $\al<\al_2$ the function $\be\in U_{\be_0}\mapsto \la_\be\in V_{\be_0}$ is $\al$-H\"older continuous from right at $\be_0$.

(3) Let 
\[\alpha_0=\frac{1}{M}\Biggl(\frac{\log|\be_0\la_{\be_0}|}{\log\be_0}\Biggr).\]
Then for any $\al>\al_0$ the function 
$\be\in U_{\be_0}\mapsto \la_\be\in V_{\be_0}$ is \textbf{not} $\al$-H\"older continuous at $\be_0$. In particular, it is non-differentiable at $\be_0$. 
\end{theorem}

If $\be_0>1$ is a Parry number, i.e., $\be_0>1$ is a number such that the orbit $\{\tau_{\be_0}^n(1)\}_{n=1}^\infty$ of $1$ is a finite set, we can easily see that 
the quantities $\displaystyle{\lim_{x\to1}\tau^n_{\be_0}(x)}$ and $1-\tau_{\be_0}^n(1)$ are bounded below by some positive constant for any $n\geq1$. 
This shows $\gamma_1=\gamma_2=1$, which yields $\al_0=\al_1=\al_2$, where the constants are those as defined in Theorem \ref{Main C}. Hence we have the following corollary. 

\begin{corollary}
    If $\be_0$ is a Parry number, the function $\be\in U_{\be_0}\mapsto \la_\be\in V_{\be_0}$ is $\al$-H\"older continuous at $\be_0$ if $\al<\al_0$ and is not $\al$-H\"older continuous at $\be_0$ if $\al>\al_0$, where $\al_0$ is the constant as defined in Theorem \ref{Main C} (3). 
\end{corollary}
    
\begin{remark}

(1) In \cite{Fl-La-Po}, Flatto et.al. showed that the modulus $M(\be)$ of the second eigenvalue(s) is continuous as a function of $\be$ (Theorem 3.1 in \cite{Fl-La-Po}), which is a consequence of Lemma \ref{lemma1-4}. From numerical experiments in case of $1<\be<2$ (see Figure 1 in \cite{Fl-La-Po}), they conjectured that $M(\be)$ is nowhere differentiable for $\be\in(1,\infty)$. 
The statement of the above theorem (3) may support this conjecture, while our result does not directly provide a positive answer to this conjecture since the non-differentiability of each second largest eigenvalue as a function $\be$ does not imply that of its modulus 
as a function of $\be$ in general. 

(2) In the context of open dynamical systems, Carminati and Tiozzo in \cite{Ca-Ti} gave the value of the Hausdorff dimension of the set $K(t)$ of points whose orbit for the doubling map on the circle does not intersect the interval $(0,t)$ using the fact that the topological entropy of the map on $K(t)$ is given by a positive zero of a certain power series defined by the kneading sequence of $t$. Our strategy for the proof here is similar to that in \cite{Ca-Ti}, although the calculations for the H\"older exponent of the function $\be\in U_{\be_0}\mapsto\la_{\be}\in V_{\be_0}$ is more involved due to the fact that $\la_{\be_0}$ is not necessary a positive number. 
\end{remark}

For the proof of Theorem \ref{Main C}, we need the following lemma:

\begin{lemma}\label{4-1}
Let $\la_{\be_0}$ be a non-leading eigenvalue of $\cL_{\be_0}$ with multiplicity $M\geq1$. Then we have the following.

(1) There is an analytic function $g_{\be_0}(z)$ defined on some neighborhood $N_{\be_0}$ of $\la_{\be_0}$ such that 
\[1-\phi_{\be_0}(z^{-1})=(z-\la_{\be_0}^{})^M g_{\be_0}(z)\]
and $g_{\be_0}(z)\neq0$ for $z\in N_{\be_0}$.

(2) There is an analytic function $\wh{g}_{\be_0}(z)$ defined on some neighborhood $\wh{N}_{\be_0}$ of $\la_{\be_0}$ such that 
\[1-\wh{\phi}_{\be_0}(z^{-1})=(z-\la_{\be_0}^{})^M \wh{g}_{\be_0}(z)\]
and $\wh{g}_{\be_0}(z)\neq0$ for $z\in N_{\be_0}$

\end{lemma}


\begin{proof}
(1) Since $\la_{\be_0}^{-1}$ is a zero of $1-\phi_{\be_0}(z)$ with multiplicity $M\geq1$, there is an analytic function $f_{\be_0}(z)$ defined on some neighborhood $\cN_{\be_0}$ of $\la_{\be_0}^{-1}$ such that 
\[1-\phi_{\be_0}(z)=(z-\la_{\be_0}^{-1})^M f_{\be_0}(z)\]
and $|f_{\be_0}(z)|>0$ if $z\in\cN_{\be_0}$. For $z\neq0$ with $z^{-1}\in \cN_{\be_0}$ we have
\begin{align*}
    1-\phi_{\be_0}(z^{-1})&=(z^{-1}-\la_{\be_0}^{-1})^M f_{\be_0}(z^{-1}) \\
    &=(z-\la_{\be_0})^M\frac{(-1)^M}{(z\la_{\be_0})^M} f_{\be_0}(z^{-1}).
\end{align*}
By taking $g_{\be_0}(z)=(-1)^M f_{\be_0}(z^{-1})/(z\la_{\be_0})^M$ and $N_{\be_0}=\{z^{-1}: z\in \cN_{\be_0}\setminus \{0\}\}$, we have the conclusion.

(2) By replacing $\phi_{\be_0}(z), g_{\be_0}(z)$ and $N_{\be_0}$ to $\wh{\phi}_{\be_0}(z), \wh{g}_{\be_0}(z)$ and $\wh{N}_{\be_0}$ respectively in the above calculations, we have the conclusion as in the same way of the proof of $(1)$. 
\end{proof}

\begin{proof}[Proof of Theorem \ref{Main C}]

(1) Let $\la_{\be_0}$ be a non-leading eigenvalue of $\cL_{\be_0}$ with multiplicity $M\geq1$. Since the map $\be\mapsto\{d_n(\be,1)\}_{n=1}^\infty$ is strictly increasing and left-continuous by Proposition 2.3 in \cite{de-Ko-Lo}, for $\be$ sufficiently close to $\be_0$ with $\be<\be_0$ there is an positive integer $N\geq1$ such that 
$\{d_n(\be,1)\}_{n=1}^N
=\{d_n(\be_0,1)\}_{n=1}^N$ and 
$d_{N+1}(\be,1)<d_{N+1}(\be_0,1)$.
By Lemma \ref{4-1} (2) there is an analytic function $\wh{g}_{\be_0}(z)$ defined on some neighborhood $\wh{N}_{\be_0}$ of $\la_{\be_0}$ such that 
\begin{equation}\label{e4-1}
1-\wh{\phi}_{\be_0}(z^{-1})=(z-\la_{\be_0}^{})^M \wh{g}_{\be_0}(z)
\end{equation}
and $|\wh{g}_{\be_0}(z)|>0$ for $z\in \wh{N}_{\be_0}$.

Take $\vep>0$ so small that $|\be_0\la_{\be_0}|>(1+\vep)^2$. 
Assume that $\be<\be_0$ is sufficiently close to $\be_0$ so that
$|\be_0\la_{\be_0}|/|\be\la_\be|<1+\vep$ and $|\wh{g}_{\be_0}(\la_{\be_0})|/|\wh{g}_{\be_0}(\la_{\be})|<1+\vep$. In one hand, since 
\begin{align*}
1-\wh{\phi}_{\be_0}(\la_{\be}^{-1})
&=\wh{\phi}_{\be}(\la_{\be}^{-1})-\wh{\phi}_{\be_0}(\la_{\be}^{-1}) \\
&=\sum_{n=1}^\infty\frac{d_n(\be,1)}{\be_{}^n\la_{\be_{}}^{n}}-\sum_{n=1}^\infty\frac{d_n(\be_{0},1)}{\be_{0}^n\la_{\be_{}}^{n}} \\
&=\sum_{n=1}^N\frac{d_n(\be,1)}{\la_{\be_{}}^{n}}\Biggl(\frac{1}{\be^n}-\frac{1}{\be_0^n}\Biggr)
+\sum_{n=N+1}^\infty\frac{d_n(\be,1)}{\be_{}^n\la_{\be_{}}^{n}}
-\sum_{n=N+1}^\infty\frac{d_n(\be_0,1)}{\be_{0}^n\la_{\be_{}}^{n}},
\end{align*}
together with the equation (\ref{e4-1}), we have
\begin{equation}\label{e4-2}
    |\la_{\be}-\la_{\be_0}|^M\leq
    C_1(\be_0-\be)+C_2\Biggl(\frac{1+\vep}{|\be_0\la_{\be_0}|}\Biggr)^{N+1}
\end{equation}
where $C_1$ and $C_2$ are the constants which are independent of $\be$ given by
\[C_1=\frac{1+\vep}{|\wh{g}_{\be_0}(\la_{\be_0})|}\sum_{n=1}^\infty\ \frac{n[\be_0]}{\be_0}\Biggl(\frac{(1+\vep)^2}{|\be_0\la_{\be_0}|}\Biggr)^n \text{and }
C_2=\frac{1+\vep}{|\wh{g}_{\be_0}(\la_{\be_0})|}\frac{2[\be_0]}{1-(1+\vep)/|\be_0\la_{\be_0}|}.\]

On the other hand, 
\begin{align*}
    \be_0-\be
    &=\sum_{n=1}^\infty \frac{d_n(\be_0,1)}{\be_0^{n-1}}-\sum_{n=1}^\infty \frac{d_n(\be,1)}{\be^{n-1}} \\
    &=\sum_{n=1}^\infty \frac{d_n(\be_0,1)}{\be_0^{n-1}}
    -\sum_{n=1}^\infty \frac{d_n(\be,1)}{\be_0^{n-1}}
    +\sum_{n=1}^\infty \frac{d_n(\be,1)}{\be_0^{n-1}}
    -\sum_{n=1}^\infty \frac{d_n(\be,1)}{\be^{n-1}} \\
    &=\sum_{n=N+1}^\infty \frac{d_n(\be_0,1)}{\be_0^{n-1}}
    -\sum_{n=N+1}^\infty \frac{d_n(\be,1)}{\be_0^{n-1}}
    +\sum_{n=2}^\infty d_{n}(\be,1)\frac{\be^{n-1}-\be_0^{n-1}}{\be_0^{n-1}\be^{n-1}},
\end{align*}
which yields 
\begin{equation}\label{e4-3}
\begin{aligned}
    \be_0-\be
    &\geq\Biggl(\sum_{n=2}^\infty d_n(\be,1)\frac{(n-1)\be_0^{n-2}}{\be^{n-1}\be_0^{n-1}}\Biggr)^{-1}\Biggl(\lim_{x\to1}\tau_{\be_0}^N(x)-\frac{d_{N+1}(\be_0,1)}{\be_0}\Biggr)\frac{1}{\be_0^{N-1}} \\
    &\geq C_3^{-1}\cdot\lim_{x\to1}\tau_{\be_0}^{N+1}(x)\cdot\frac{1}{\be_0^{N+1}},
\end{aligned}
\end{equation}
where 
\[C_3=\sum_{n=1}^\infty n [\be_0]\Biggl(\frac{1+\vep}{\be_0}\Biggr)^n_.\]
Together with the above inequalities, we obtain 
\begin{equation}\label{inequality1}
\begin{aligned}
    \Biggl|\frac{\la_{\be_0}-\la_\be}{(\be_0-\be)^\al}\Biggr|^M
    &\leq C_1 (\be_0-\be)^{1-M\al} \\
    &+C_2C_{3}^{M\al}\Biggl(\frac{1+\vep}{|\be_0\la_{\be_0}|}\Biggr)^{N+1}
    \Biggl(\lim_{x\to1}\tau_{\be_0}^{N+1}(x)\cdot\frac{1}{\be_0^{N+1}}\Biggr)^{-M\al}
 \end{aligned}
\end{equation}
for $0<\al<1/M$. 
Since $1-M\al>0$ we know that $(\be_0-\be)^{1-M\al}\to0$ as $\be\to\be_0$, which guarantees that the first term of the right side of the inequality (\ref{inequality1}) is bounded above by some positive constant independent of $\be$ sufficiently close to $\be_0$. 
Concerning the second term of the right side of the inequality (\ref{inequality1}), by the definition of $\al_1$, we have
\begin{align*}
    &\frac{1}{N+1}\log\Biggl(\Biggl(\frac{1+\vep}{|\be_0\la_{\be_0}|}\Biggr)^{N+1}
    \Biggl(\lim_{x\to1}\tau_{\be_0}^{N+1}(x)\cdot\frac{1}{\be_0^{N+1}}\Biggr)^{-M\al}\Biggr) \\
    &=M\al \log\be_0
    +\frac{M\al}{N+1}\log(\lim_{x\to1}\tau_{\be_0}^{N+1}(x))^{-1}+\log(1+\vep)-\log|\be_0\la_{\be_0}| \\
    &=M(\al-\al_1)\log\be_0+\frac{M(\al-\al_1)}{N+1}\log(\lim_{x\to1}\tau_{\be_0}^{N+1}(x))^{-1} \\
    &+M\al_1\log\be_0+\frac{M\al_1}{N+1}\log(\lim_{x\to1}\tau_{\be_0}^{N+1}(x))^{-1}
    +\log(1+\vep)-\log|\be_0\la_{\be_0}| \\
    &=-M(\al_1-\al)\log\be_0-\frac{M(\al_1-\al)}{N+1}\log(\lim_{x\to1}\tau_{\be_0}^{N+1}(x))^{-1} \\
    &+\frac{(N+1)^{-1}\log(\lim_{x\to1}\tau_{\be_0}^{N+1}(x))^{-1}-\log\gamma_1^{-1}}{\log\be_0+\log\gamma_1^{-1}}\log|\be\la_{\be_0}|+\log(1+\vep) \\
    &<0
\end{align*}
for $0<\al<\al_1$ if $\be$ is sufficiently close to $\be_0$ and $\vep$ is sufficiently small since $\la_\be\to\la_{\be_0}$ and $N \to$ $\infty$ as $\be\to\be_0$.
This shows that the second term of the right side of the inequality (\ref{inequality1}) is bounded above by some constant if $\be$ is sufficiently close to $\be_0$, which concludes that there is a positive constant $C'$ such that 
\begin{equation}\label{(1)c}
\Biggl|\frac{\la_{\be_0}-\la_\be}{(\be_0-\be)^\al}\Biggr|^M
    \leq C'
\end{equation}
for any $\be$ sufficiently close to $\be_0$. This gives the conclusion.

(2) We have the desired result in the same way of (1). Let $\la_{\be_0}$ be a non-leading eigenvalue of $\cL_{\be_0}$ with multiplicity $M\geq1$.
Since the map $\be\mapsto\{a_n(\be,1)\}_{n=1}^\infty$ is strictly increasing and right-continuous by Proposition 2.2 in \cite{de-Ko-Lo}, for $\be\in U_{\be_0}$ sufficiently close to $\be_0$ with $\be>\be_0$ there is an positive integer $N\geq1$ such that 
$\{a_n(\be,1)\}_{n=1}^N
=\{a_n(\be_0,1)\}_{n=1}^N$ and 
$a_{N+1}(\be,1)>a_{N+1}(\be_0,1)$.
By Lemma \ref{4-1} (1) there is an analytic function $g_{\be_0}(z)$ defined on some neighborhood $N_{\be_0}$ of $\la_{\be_0}$ such that 
\[1-\phi_{\be_0}(z^{-1})=(z-\la_{\be_0}^{})^M g_{\be_0}(z)\]
and $|g_{\be_0}(z)|>0$ $z\in N_{\be_0}$.

Take $\vep>0$ so small that $|\be_0\la_{\be_0}|>(1+\vep)^2$. 
Assume that $\be>\be_0$ is sufficiently close to $\be_0$ so that
$\be/\be_0$, $|\be\la_{\be}|/|\be_{0}\la_{\be_{0}}|<1+\vep$ and $|g_{\be_0}(\la_{\be})|/|g_{\be_0}(\la_{\be_0})|<1+\vep$. By replacing $\wh{\phi}_{\be_0}$ to $\phi_{\be_0}$ in the calculations for obtaining the inequality (\ref{e4-2}), we have
\begin{equation}\label{e4-3-1}
    |\la_{\be}-\la_{\be_0}|^M\leq
    D_1(\be-\be_0)+D_2\Biggl(\frac{1+\vep}{|\be_0\la_{\be_0}|}\Biggr)^{N+1}
\end{equation}
where $D_1$ and $D_2$ are the constants which are independent of $\be$ given by 
\[D_1=\frac{1+\vep}{|g_{\be_0}(\la_{\be_0})|}\sum_{n=1}^\infty \frac{n[\be_0]}{\be_0}\Biggl(\frac{(1+\vep)^2}{|\be_0\la_{\be_0}|}\Biggr)^n \text{ and }
D_2=\frac{1+\vep}{|g_{\be_0}(\la_{\be_0})|}\frac{2[\be_0]}{1-(1+\vep)/|\be_0\la_{\be_0}|}.\]
Similarly, by replacing the sequences $\{d_n(\be,1)\}_{n=1}^\infty$ and $\{d_n(\be_0,1)\}_{n=1}^\infty$ to the sequences $\{a_n(\be,1)\}_{n=1}^\infty$ and $\{a_n(\be_0,1)\}_{n=1}^\infty$ respectively in the calculations for obtaining the inequality (\ref{e4-3}), we have
\begin{equation}\label{e4-4}
\begin{split}
    \be-\be_0
    &\geq\Biggl(\sum_{n=2}^\infty a_n(\be_0,1)\frac{(n-1)\be^{n-2}}{\be^{n-1}\be_0^{n-1}}\Biggr)^{-1}\Biggl(\sum_{n=N+1}^\infty\frac{a_n(\be,1)}{\be^{n-1}}-\sum_{n=N+1}^\infty\frac{a_n(\be_0,1)}{\be^{n-1}}\Biggr) \\
    &\geq\Biggl(\sum_{n=2}^\infty [\be_0]\frac{(n-1)}{\be_0^{n}}\Biggr)^{-1}\bigl(\tau_\be^N(1)-\tau_{\be_0}^N(1)\bigr)\frac{1}{\be^{N}}.
    \end{split}
\end{equation}
Since 
\begin{align*}
    &\tau_\be^N(1)-\tau_{\be_0}^N(1) \\
    &\geq \frac{a_{N+1}(\be_0,1)+1}{\be}-\tau_{\be_0}^N(1) \\
    &\geq\frac{a_{N+1}(\be_0,1)+1}{\be}-\frac{a_{N+1}(\be_0,1)+1}{\be_0}+\frac{a_{N+1}(\be_0,1)+1}{\be_0}-\tau_{\be_0}^N(1) \\
    &=-\frac{a_{N+1}(\be_0,1)+1}{\be\be_0}(\be-\be_0)
    +\bigl(1-\tau_{\be_0}^{N+1}(1)\bigr)\cdot\frac{1}{\be_0},
\end{align*}
together with the inequality (\ref{e4-4}), we obtain
\[\be-\be_0\geq D_3^{-1} (1-\tau_{\be_0}^{N+1}(1))\cdot\frac{1}{\be_0^{N+1}(1+\vep)^{N+1}},\]
where $D_3$ is a positive constant independent of any $\be$ sufficiently close to $\be_0$. This and the inequality (\ref{e4-3-1}) give the inequality:
\begin{equation}\label{inequality2}
\begin{aligned}
\Biggl|\frac{\la_{\be_0}-\la_\be}{(\be-\be_0)^\al}\Biggr|^M
    &\leq D_1(\be-\be_0)^{1-M\al}\\ 
    &+D_2D_{3}^{M\al}\Biggl(\frac{1+\vep}{|\be_0\la_{\be_0}|}\Biggr)^{N+1}
    \Biggl((1-\tau_{\be_0}^{N+1}(1))\cdot\frac{1}{(\be_0(1+\vep))^{N+1}}\Biggr)_.^{-M\al}
\end{aligned}
\end{equation}
Since $1-M\al>0$ we know that $(\be-\be_0)^{1-M\al}\to0$ as $\be\to\be_0$, which ensures that the first term of the right side of the above inequality (\ref{inequality2}) is bounded above by some positive constant independent of $\be$ sufficiently close to $\be_0$.  
As in the same way of the estimation of the second term of (\ref{inequality1}), we have that 
\begin{align*}
    &\frac{1}{N+1}\log\Biggl(\Biggl(\frac{1+\vep}{|\be_0\la_{\be_0}|}\Biggr)^{N+1}
    \Biggl((1-\tau_{\be_0}^{N+1}(1))\cdot\frac{1}{(\be_0(1+\vep))^{N+1}}\Biggr)^{-M\al}\Biggr) \\
    &=M\al \log\be_0
    +\frac{M\al}{N+1}\log(1-\tau_{\be_0}^{N+1}(1))^{-1}+(M\al+1)\log(1+\vep)-\log|\be_0\la_{\be_0}| \\
    &=M(\al-\al_2)\log\be_0+\frac{M(\al-\al_2)}{N+1}\log(1-\tau_{\be_0}^{N+1}(1))^{-1} \\
    &+M\al_1\log\be_0+\frac{M\al_1}{N+1}\log(\lim_{x\to1}\tau_{\be_0}^{N+1}(x))^{-1}
    +(M\al+1)\log(1+\vep)-\log|\be_0\la_{\be_0}| \\
    &=-M(\al_2-\al)\log\be_0-\frac{M(\al_2-\al)}{N+1}\log(1-\tau_{\be_0}^{N+1}(1))^{-1} \\
    &+\frac{(N+1)^{-1}\log(1-\tau_{\be_0}^{N+1}(1))^{-1}-\log\gamma_2^{-1}}{\log\be_0+\log\gamma_2^{-1}}\log|\be_0\la_{\be_0}|+(M\al+1)\log(1+\vep) \\
    &<0
\end{align*}
if $\be$ is sufficiently close to $\be_0$ and $\vep$ is sufficiently small since $\la_\be\to\la_{\be_0}$ and $N \to$ $\infty$ as $\be\to\be_0$. This shows that there is some positive constant $D'$ satisfying 
\begin{equation}\label{(2)c}
\Biggl|\frac{\la_{\be_0}-\la_\be}{(\be_0-\be)^\al}\Biggr|^M
    \leq D'
\end{equation}
for any $\be$ sufficiently close to $\be_0$, which gives the conclusion. 

(3) Let $\{l(N)\}_{N=1}^\infty$ be a sequence of positive integers such that $l(N)<l(N+1)$ and $d_{l(N)}(\be_0,1)>0$ for any $N\geq1$ satisfying  $d_i(\be_0,1)=0$ if $i\notin\{l(N)\}_{N=1}^\infty$. Note that we can take such a sequence since for every $n\geq1$ there is $n_0$ with $n_0\geq n$ such that $d_{n_0}(\be_0,1)>0$ by the definition of the quasi greedy expansion of $1$. 

For each $N\geq1$ take the positive number $\be_N>1$ satisfying
\[1=\sum_{n=1}^{l(N)}\frac{d_n(\be_0,1)}{\be_N^n}.\]
By Proposition \ref{Parry2}, we know that each $\be_N$ is a simple Parry number satisfying $\{a_n(\be_N,1)\}_{n=1}^{l(N)}=\{d_n(\be_0,1)\}_{n=1}^{l(N)}$ and $a_n(\be_N,1)=0$ for $n\geq l(N)+1$.
Note that $\be_N<\be_{N+1}$ for $N\geq1$ and $\displaystyle{\lim_{N\to\infty}}\be_N=\be_0$. Since
\[\be_0-\be_N=\sum_{n=l(N)+1}^\infty\frac{d_{n}(\be_0,1)}{\be_0^{n-1}}-\sum_{n=2}^{l(N)} d_n(\be_0,1)\frac{\be_0^{n-1}-\be_N^{n-1}}{\be_{0}^{n-1}\be_N^{n-1}},\]
we have
\begin{equation}\label{bunbo}
    \begin{split}
        \be_0-\be_N\leq \sum_{n=l(N)+1}^\infty\frac{d_{n}(\be_0,1)}{\be_0^{n-1}}\leq E_1 \frac{1}{\be_0^{l(N+1)}},
    \end{split}
\end{equation}
where
\[E_1=\frac{[\be_0]}{1-1/\be_0}.\]
By Lemma \ref{4-1} (2), we know that there is an analytic function $\wh{g}_{\be_0}(z)$ defined on some neighborhood of $\la_{\be_0}$  such that
\[(\la_{\be_0}-\la_{\be_N}^{})^M= \frac{1}{\wh{g}_{\be_0}(\la_{\be_N})}(1-\wh{\phi}_{\be_0}(\la_{\be_N}^{-1}))\]
and $|\wh{g}_{\be_0}(\la_{\be_0})|>0$ for any $\be$ sufficiently close to $\be_0$. 
Since
\begin{align*}
1-\wh{\phi}_{\be_0}(\la_{\be_N}^{-1})
&=\phi_{\be_N}(\la_{\be_N}^{-1})-\wh{\phi}_{\be_0}(\la_{\be_N}^{-1}) \\
&=\sum_{n=1}^{l(N)}\frac{d_n(\be_0,1)}{\be_{N}^n\la_{\be_{N}}^{n}}-\sum_{n=1}^\infty\frac{d_n(\be_{0},1)}{\be_{0}^n\la_{\be_{N}}^{n}} \\
&=\sum_{n=1}^{l(N)}\frac{d_n(\be_0,1)}{\la_{\be_{N}}^{n}}\Biggl(\frac{1}{\be_N^n}-\frac{1}{\be_0^n}\Biggr)
-\sum_{n=l(N)+1}^\infty\frac{d_n(\be_0,1)}{\be_{0}^n\la_{\be_{N}}^{n}}
\end{align*}
for $N\geq1$, we obtain
\[|\la_{\be_N}-\la_{\be_0}|^M\geq \frac{1}{|\wh{g}_{\be_0}(\la_{\be_N})|}\Biggl(\Biggl|\sum_{n=l(N+1)}^\infty\frac{d_n(\be_0,1)}{\be_0^n\la_{\be_N}^n}\Biggr|-\sum_{n=1}^\infty\frac{n[\be_0]}{\be_0|\be_N\la_{\be_N}|^n}(\be_0-\be_N)\Biggr)\]
for $N\geq1$. Hence
\begin{equation}\label{e4-7}
\begin{split}
\Biggl|\frac{\la_{\be_0}-\la_{\be_N}}{(\be_N-\be_0)^\al}\Biggr|^M
&\geq\frac{1}{|\wh{g}_{\be_0}(\la_{\be_N})|}\Biggl(
\frac{1}{|\be_0-\be_N|^{M\al}}\Biggl|\sum_{n=l(N+1)}^\infty\frac{d_n(\be_0,1)}{\be_0^n\la_{\be_N}^n}\Biggr|\\
&-\sum_{n=1}^\infty\frac{n}{|\be_N\la_{\be_N}|^n}(\be_0-\be_N)^{1-M\al}\Biggr)_.
\end{split}
\end{equation}
By the fact that $1-M\al>0$, the second term of the right side of the above inequality is bounded below by some constant independent 
of $\be_N$ for sufficiently large $N\geq1$. 

In the following, we show that there is an increasing sequence
$\{N_i\}_{i=1}^\infty$ of positive integers such that 
\[\frac{1}{|\be_0-\be_{N_i}|^{M\al}}\Biggl|\sum_{n=l(N_i+1)}^\infty\frac{d_n(\be_0,1)}{\be_0^n\la_{\be_{N_i}}^n}\Biggr|\to\infty\]
as $i\to\infty$ if $\al>\al_0$. 
From the inequality (\ref{e4-7}), this yields that 
\[\Biggl|\frac{\la_{\be_0}-\la_{\be_{N_i}}}{(\be_{N_i}-\be_0)^\al}\Biggr|^M
\to\infty\]
as $i\to\infty$, since $|\wh{g}_{\be_0}(\la_{\be_{N_i}})|\to |\wh{g}_{\be_0}(\la_{\be_0})|>0$ as $i\to\infty$, which gives the conclusion.

\begin{sublemma}\label{sublem}
    There are a positive constant $C$ and a strictly increasing sequence $\{l^{'}(M)\}_{M=1}^\infty$ of positive integers with $d_{l'(M)}(\be_0,1)>0$ for $M\geq1$ such that 
    \[\Biggl|\sum_{n=1}^\infty\frac{d_{n-1+l'(M)}(\be_0,1)}{\be_0^n\la_{\be_0}^n}\Biggr|\geq C\]
    for any $M\geq1$. 
\end{sublemma}

\begin{proof}
Set 
\[\tau_j=\sum_{n=1}^\infty\frac{d_{n-1+j}(\be_0,1)}{\be_0^n\la_{\be_0}^n}\]
for $j\geq1$. First we show that there are positive constant $C$ and a strictly increasing sequence $\{j_k\}_{k=1}^\infty$ of positive integers such that $|\tau_{j_k}|>C$ for $k\geq1$.
Assume that $|\tau_j|\to0$ as $j\to\infty$. By definition we know that $\be_0\la_{\be_0}\tau_j-d_{j}(\be_0,1)=\tau_{j+1}$ for $j\geq1$. Then $d_{j}(\be_0,1)\leq |\be_0\la_{\be_0}||\tau_j|+|\tau_{j+1}|\to0$ as $j\to\infty$, which ensures that there is a positive integer $j_0$ such that $d_{j}(\be_0,1)=0$ for any $j\geq j_0$. By the definition of the quasi-greedy expansion of $1$, however, we know that for any $n\geq1$ there is $n_0\geq n$ such that $d_{n_0}(\be_0,1)>0$, which shows the contradiction. 

If there is a strictly increasing sequence of positive integers $\{k_M\}_{M=1}^\infty$ satisfying $d_{j_{k_M}}(\be_0,1)>0$ for any $M\geq1$, we have the conclusion by taking $l^{'}(M)=j_{k_M}$ for $M\geq1$. 

If there is $k_0$ such that $d_{j_k}(\be_0,1)=0$ for any $k\geq k_0$, we can take the desired sequence as follows. 
Let $l^{'}(1)=\min\{n>j_{k_0}; d_n(\be_0,1)>0\}$. Since $d_n(\be_0,1)=0$ for $j_{k_0}\leq n<l^{'}(1)$, we have 
\[\tau_{l^{'}(1)}=|\be_0\la_{\be_0}|^{l^{'}(1)-j_{k_0}}\tau_{j_{k_0}}\geq|\be_0\la_{\be_0}|^{l^{'}(1)-j_{k_0}}C\geq C\]
since $|\be_0\la_{\be_0}|>1$. Set $k_1$ so that $j_{k_1}=\min\{n>l^{'}(1); n\in\{j_k\}_{k=1}^\infty\}$. We know that $d_{j_{m}}(\be_0,1)=0$ for $m\geq k_1$. Let $l^{'}(2)=\min\{n>j_{k_1}\ ;\ d_n(\be_0,1)>0\}$.
Similarly to the calculation for $\tau_{l^{'}(1)}$, we have
\[\tau_{l^{'}(2)}=|\be_0\la_{\be_0}|^{l^{'}(2)-j_{k_1}}\tau_{j_{k_1}}\geq|\be_0\la_{\be_0}|^{l^{'}(2)-j_{k_1}}C>C.\]
Inductively, by taking 
\[l^{'}(M)=\min\{n>j_{k_{M-1}}\ ;\ d_n(\be_0,1)>0\}\]
and $\{k_M\}_{M=1}^\infty$ so that 
\[j_{k_M}=\min\{n>l^{'}(M)\ ;\ n\in\{j_k\}_{k=1}^\infty\}\]
for $M\geq1$, we have that $d_{l^{'}(M)}(\be_0,1)>0$ and $\tau_{l^{'}(M)}>C$ for $M\geq1$, which finishes the proof.
\end{proof}

\noindent
{\it Proof of Theorem \ref{Main C}, continued}

\noindent

Let $\{l^{'}(M)\}_{M=1}^\infty$ be the sequence as in the above sublemma. For $M\geq1$ we define the formal power series $P_M(z)$ by 
\[P_M(z)=\sum_{n=1}^\infty \frac{d_{n-1+l^{'}(M)}(\be_0,1)}{\be_0^n\la_{\be_0}^n}z^n.\]
We note that the convergence radius of $P_M(z)$ is at least $|\be_0\la_{\be_0}|>1$. Let $\vep$ be any positive number with $1+\vep<|\be_0\la_{\be_0}|$. Since 
\[|P_M(z)|\leq \sum_{n=1}^\infty[\be_0]\Biggl(\frac{|z|}{|\be_0\la_{\be_0}|}\Biggr)^n
\leq \frac{[\be_0]}{1-(1+\vep)/|\be_0\la_{\be_0}|}
\]
for any $M\geq1$ and $z\in\C$ with $|z|\leq 1+\vep$, we have that $\{P_M(z)\}_{M=1}^\infty$ is a normal family on $\{z\in\C; |z|<1+\vep\}$. 
Together with the fact that there is a positive constant $C$ such that $|P_M(1)|>C$ for any $M\geq1$ by Sublemma \ref{sublem}, we have that there are a positive number $\delta'>0$ and a strictly increasing sequence $\{M_i\}_{i=1}^\infty$ of positive integers such that 
$|P_{M_i}(z)|\geq C/2$ for any $z\in\C$ with $|z-1|<\delta'$.

Take an increasing sequence $\{N_i\}_{i=1}^\infty$ of positive integers so that $l(N_i)=l'(M_i)$ for $i\geq1$. Let $\delta$ be a positive number so small that $\delta<\min\{\delta',\vep\}$ and $\be_0^{M\al_0}(1-\delta)/|\be_0\la_{\be_0}|>1$. For $i\geq1$ so large as $|1-(\la_{\be_0}/\la_{\be_{N_i}})|<\delta$, we have 
\begin{align*}
    \Biggl|\sum_{n=l(N_i+1)}^\infty\frac{d_n(\be_0,1)}{\be_0^n\la_{\be_{N_i}}^n}\Biggr|
    &=\Biggl|\sum_{n=1}^\infty \frac{d_{n-1+l(N_i+1)}(\be_0,1)}{\be_0^n\la_{\be_0}^n}\Bigl(\frac{\la_{\be_0}}{\la_{\be_{N_i}}}\Bigr)^n\Biggr|\Biggl|\frac{1}{\be_0\la_{\be_{N_i}}}\Biggr|^{l(N_i+1)-1} \\
    &=\Biggl|P_{N_i+1}\Bigl(\frac{\la_{\be_0}}{\la_{\be_{N_i}}}\Bigr)\Biggr| 
    \Biggl|\frac{1}{\be_0\la_{\be_{0}}}
    \frac{\la_{\be_0}}{\la_{\be_{N_i}}}\Biggr|^{l(N_i+1)-1} \\
    &\geq \frac{C}{2}\Biggl(\frac{1-\delta}{|\be_0\la_{\be_0}|}\Biggr)^{l(N_i+1)-1}_.
\end{align*}
Together with the inequality ($\ref{bunbo}$), we obtain
\begin{equation}\label{e4-11}
    \frac{1}{|\be_0-\be_{N_i}|^{M\al}}
    \Biggl|\sum_{n=l(N_i+1)}^\infty\frac{d_n(\be_0,1)}{\be_0^n\la_{\be_{N_i}}^n}\Biggr|
    \geq \frac{C}{2}\frac{1}{E_1^{M\al}}\Biggl(\frac{(1-\delta)\be_0^{M\al}}{|\be_0\la_{\be_0}|}\Biggr)^{l(N_i+1)-1}
\end{equation}
for $i\geq1$ and $\al>0$. Since 
\[\frac{(1-\delta)\be_0^{M\al}}{|\be_0\la_{\be_0}|}=
\frac{(1-\delta)\be_0^{M\al_0}}{|\be_0\la_{\be_0}|}\be_0^{M(\al-\al_0)}>1\]
for $\al>\al_0$, we have that the right side of the inequality (\ref{e4-11}) goes to infinity as $i\to\infty$ if $\al>\al_0$, which ends the proof. 
\end{proof}

\begin{remark}
In the proof of Theorem \ref{Main C} (3), we in fact show that the function $\be\in U_{\be_0}\mapsto\la_{\be}\in V_{\be_0}$ is not $\al$-H\"older continuous from left at $\be_0$ if $\al>\al_0$. Establishing the same result for right continuity seems to be a more delicate problem and 
it may depend on the coefficient sequence of the quasi-greedy expansion of $1$. 
\end{remark}

\end{section}

\begin{section}{Properties of eigenfunctionals}

In this section, we relate the value of every eigenfunctional of $\cL_\be$ applied to the indicator function of an interval to the greedy expansion (Proposition \ref{key2}) and give its three applications. 
The following lemma is a key tool for obtaining the main results in this section.

\begin{lemma}\label{key1}
Let $n\geq 0$ be a non-negative integer. Then we have
\begin{equation*}
\cL_\be \1_{[0,\tau_\be^n(x)]}=\frac{a_{n+1}(\be,x)}{\be}\1_{[0,1]}+\frac{1}{\be}\1_{[0,\tau_\be^{n+1}(x)]},
\end{equation*}
where $\1_A$ denote the indicator function of $A$. 
\end{lemma}

\begin{proof}
By the definition of $\cL_{\be}$, for $a\in[0,1]\setminus\{\tau^{n+1}_\be(x)\}$ we have
\[\begin{split}
\be\cL_{\be}\1_{[0,\tau_{\be}^{n}(x)]}(a)
&=\#\{b\in[0,\tau_{\be}^n(x)]\ ;\ a=\tau_{\be}(b)\} \\
&=
\begin{cases}
a_{n+1}(\be,x)+1 &(a<\tau_\be^{n+1}(x)), \\
a_{n+1}(\be,x) & (a>\tau_\be^{n+1}(x)), \\
\end{cases} 
\end{split}\]
since the map $\tau_{\be}$ is bijective on $[i/\be,(i+1)/\be)$ for $0\leq i \leq [\beta]-1$.
This shows that 
\[\cL_\be \1_{[0,\tau_\be^n(x)]}=\frac{a_{n+1}(\be,x)}{\be}\1_{[0,1]}+\frac{1}{\be}\1_{[0,\tau_\be^{n+1}(x)]},\]
which finishes the proof.
\end{proof}

\begin{proposition}\label{key2}
Let $\la\in\C$ be an isolated eigenvalue of the Perron-Frobenius operator $\cL_{\be}$ with $|\la|>1/\be$ and let $\nu\in\BV^*$ be an eigenfunctional corresponding to $\la$. Then for $x\in[0,1]$ we have
\[\nu(\1_{[0,x]})=\nu(\1_{[0,1]})\sum_{n=1}^{\infty}\frac{a_n(\be, x)}{\be^n\la^n}.\]
In addition, $\nu(\mathbf{1}_{[0,1]})\neq0$.
\end{proposition}

\begin{proof}
Since $\nu$ is an eigenfunctional of the Perron-Frobenius operator $\cL_\be$, we know that
\[\nu(f)=\frac{(\cL_\be^*\nu)(f)}{\la}
=\frac{\nu(\cL_\be f)}{\la}\]
for $f\in \BV$. 
By Lemma \ref{key1}, we have
\begin{align*}
\nu(\1_{[0, \tau_\beta^n(x)]})
&=\frac{\nu(\cL_\be \1_{[0, \tau_\beta^n(x)]})}{\la} \\
&=\frac{1}{\la}\Bigl(\nu\Bigl(\frac{a_{n+1}(\be, x)}{\be}\1_{[0,1]}+\frac{1}{\be}\1_{[0,\tau_\be^{n+1}(x)]}\Bigr)\Bigr) \\
&=\frac{a_{n+1}(\be, x)}{\be\la}\nu(\1_{[0,1]})+\frac{\nu(\1_{[0,\tau_\be^{n+1}(x)]})}{\be\la}
\end{align*}
for $n\geq0$ and $x\in[0,1]$. By using this formula inductively, we obtain
\begin{align}\label{above}
\nu(\1_{[0,x]})
&=\frac{a_{1}(\be, x)}{\be\la}\nu(\1_{[0,1]})+\frac{\nu(\1_{[0,\tau_\be^{}(x)]})}{\be\la}\\ \notag
&= \frac{a_{1}(\be, x)}{\be\la}\nu(\1_{[0,1]})+\frac{a_{2}(\be, x)}{\be^2\la^2}\nu(\1_{[0,1]})+\frac{\nu(\1_{[0,\tau_\be^{2}(x)]})}{\be^2\la^2}\\ \notag
&=\dots \\ \notag
&=\sum_{n=1}^N\frac{a_n(\be, x)}{\be^n\la^n}\nu(\1_{[0,1)})+\frac{\nu(\1_{[0,\tau_\be^{N+1}(x)]})}{\be^{N+1}\la^{N+1}}
\end{align}
for $N\geq1$. Note that $|\be\la|>1$ since $1/\be<|\la|\leq1$. Together with the fact that 
\[|\nu(\1_{[0,\tau_\be^m(x)]})|\leq||\nu||\cdot(|\1_{[0,\tau_\be^m(x)]}|_{\infty}+\var(\1_{[0,\tau_\be^m(x)]}))\leq2 ||\nu||\]
for $m\geq0$, where $||\cdot||$ denotes the operator norm on $\BV^*$,
taking $N\to\infty$ in the rightmost side of the equation (\ref{above}) gives
\begin{equation}\label{efef}
\nu(\1_{[0,x]})=\nu(\1_{[0,1]})\cdot\sum_{n=1}^{\infty}\frac{a_n(\be, x)}{\be^n\la^n},
\end{equation}
as desired. 

We shall prove that $\nu(\1_{[0,1]})\neq0$. Let $h\in \BV$ be an eigenfunction of $\cL_\be$ corresponding to $\la$. Since the geometric multiplicity of the eigenspace of $\cL_\be$ corresponding to $\la$ is equal to $1$ by Proposition \ref{gm}, we know that $\nu(h)\neq0$ for any eigenfunctional $\nu\in\BV^*\setminus\{0\}$ corresponding to $\la$. 
By Theorem 3.2 in \cite{Su2}, the eigenfunction $h$ has the form 
\[h=C \sum_{m=0}^\infty\frac{\1_{[0,\tau_\be^m(1)]}}{\be^m\la^m},\]
where $C$ is a non-zero constant. Together with the equality (\ref{efef}), we obtain 
\begin{align*}
\nu(h)
=C\cdot \sum_{m=0}^\infty\frac{\nu(\1_{[0,\tau_\be^m(1)]})}{\be^m\la^m} 
=\nu(\1_{[0,1]})\cdot C\cdot\sum_{m=0}^\infty\sum_{n=1}^\infty \frac{a_n(\be, \tau^m_{\be}(1))}{\be^{m+n}\la^{m+n}},
\end{align*}
which concludes that $\nu(\1_{[0,1]})\neq0$. 

\end{proof}
Let $\la\in\C$ be a complex number with $1/\be<|\la|\leq1$. As motivated by the formula for an eigenfunctional of $\cL_\be$ in Proposition \ref{key2}, we investigate the complex-valued function $F_\la:[0,1]\to\C$ defined by 
\begin{equation}\label{functional}
F_\la(x)=\sum_{n=1}^{\infty}\frac{a_n(\be, x)}{\be^n\la^n}
\end{equation}
for $x\in[0,1]$. The main result of this section is the following characterization for $\la\in\C$ and $F_\la$.

\begin{theorem}\label{main a}
Let $\la\in\C$ be a complex number with $1/\be<|\la|\leq1$. Then we have the following.

(1) The function $F_\la(x)$ is right-continuous on $[0,1)$.

(2) The function $F_\la(x)$ is left-continuous at any non-simple $x\in(0,1]$.

(3) The function $F_\la(x)$ is left-continuous at any simple $x\in(0,1]$ if and only if $\la$ is an isolated eigenvalue of $\cL_\be$. 

(4) If $\la$ is the leading eigenvalue $1$ of $\cL_\be$, then $F_\la(x)=x$. In particular, it is differentiable on $[0,1]$. 

(5) If $\la$ is a non-leading eigenvalue of $\cL_\be$, then $F_\la(x)$ is not Lipschitz continuous on $[0,1]$. In particular, it is nowhere differentiable on $[0,1]$. 

\end{theorem}

\begin{proof}
(1) By definition the function 
\[x\in[0,1)\mapsto a_{n}(\be, x)=[\be\tau_\be^{n-1}(x)]\in\{0,\dots,[\be]\}\]
is right continuous for each $n\geq1$. Since $\sum_{n=1}^N a_n(\be, x)/(\be\la)^n$ converges uniformly for $x\in[0,1]$ as $N\to\infty$, we have that the limit $F_\la(x)=\sum_{n=1}^\infty a_n(\be, x)/(\be\la)^n$ is also right continuous on $[0,1)$. 

(2) Assume that $x_0\in(0,1]$ is non-simple, i.e., $\tau_\be^m(x_0)\notin\{1/\be,\dots,[\be]/\be\}$ for any $m\geq0$. In this case, left continuity of all the functions 
\[x\in(0,1]\mapsto a_{n}(\be, x)=[\be\tau_\be^{n-1}(x)]\in\{0,\dots,[\be]\}\]
at $x_0$ for any $n\geq1$ gives the conclusion since $\sum_{n=1}^N a_n(\be, x)/(\be\la)^n$ converges uniformly for $x\in(0,1]$ as $N\to\infty$. This yields the limit $F_\la(x)=\sum_{n=1}^\infty a_n(\be, x)/(\be\la)^n$ is also left continuous at non-simple $x_0$. 

(3) Assume that $x_0\in[0,1]$ is simple, i.e., there is a positive integer $L(x_0)\geq1$ such that  $\tau_\be^{L(x_0)-1}(x_0)\in\{1/\be,\dots,[\be]/\be\}$. Then the $\be$-expansion of $x_0$ is finite, i.e., 
\[x_0=\sum_{n=1}^{L(x_0)}\frac{a_n(\be, x_0)}{\be^n}.\]
By the left-continuity of the digit sequence $\{d_n(\be,x)\}_{n=1}^\infty$ of the quasi-greedy expansion as noted in Section 2, the limit of the digit sequence of the greedy expansion from left is equal to that of the quasi-greedy expansion.
This yields
\[F_\la(x_0)=\sum_{n=1}^{L(x_0)}\frac{a_n(\be, x_0)}{\be^n\la^n}\]
and
\begin{align*}
\lim_{x\nearrow x_0}F_\la(x_0)
&=\sum_{n=1}^\infty\frac{d_n(\be,x_0)}{\be^n\la^n} \\
&=\sum_{n=1}^{L(x_0)}\frac{a_n(\be, x_0)}{\be^n\la^n}-\frac{1}{\be^{L(x_0)}\la^{L(x_0)}}+\frac{1}{\be^{L(x_0)}\la^{L(x_0)}}\sum_{n=1}^\infty\frac{d_n(\be, 1)}{\be^n\la^n}.
\end{align*}
Hence it is sufficient to show that 
\[-\frac{1}{\be^{L(x_0)}\la^{L(x_0)}}+\frac{1}{\be^{L(x_0)}\la^{L(x_0)}}\sum_{n=1}^\infty\frac{d_n(\be,1)}{\be^n\la^n}=0,\]
namely, 
\[1=\sum_{n=1}^\infty\frac{d_n(\be,1)}{\be^n\la^n}=\Hat{\phi}_\be(\la^{-1}).\]

Since $\la$ is an isolated eigenvalue of $\cL_\be$, by Theorem \ref{phi}, we know that $1-\phi_\be(\la^{-1})=0$. Hence $1-\Hat{\phi}_\be(\la^{-1})=0$ follows from Proposition \ref{a=d}, which gives the conclusion.


(4) Let $\la=1$. By the definition of the beta-expansion of $x\in[0,1]$, we have
\[F_1(x)=\sum_{n=1}^\infty\frac{a_n(x)}{\beta^n}=x,\]
which yields the conclusion.

(5) Let $\la$ be a non-eigenvalue of $\cL_\be$. In case of $x=0$, by taking $x_N=1/\be^N$ for $N\geq1$, we have
\[\Biggl|\frac{F_\la(x_N)-F_\la(0)}{x_N-0}\Biggr|=\be^N\Biggl|\frac{1}{\be^N\la^N}\Biggr|=\frac{1}{|\la|^N}\to\infty\]
as $N\to\infty$, which provides the conclusion at $x=0$.

Let $x_0\in(0,1]$. Assume that the function $F_\la(x)$ is Lipschitz continuous from left at $x_0$, that is, there are a positive constant $C>0$ and $\delta>0$ such that 
\[|F_\la(x)-F_\la(x_0)|\leq C|x-x_0|\] 
for any $x\in(x_0-\delta, x_0]$. 

We shall show that \[F_\la(x)=\sum_{n=1}^\infty\frac{a_n(\be,x)}{\be^n\la^n}=\sum_{n=1}^\infty\frac{d_n(\be,x)}{\be^n\la^n}\]for $x\in(0,1]$. If $x$ is non-simple, we have the result by the fact that $\{d_n(\be,x)\}_{n=1}^\infty=\{a_n(\be,x)\}_{n=1}^\infty$. If $x$ is simple, by the definition of the sequence $\{d_n(\be,x)\}_{n=1}^\infty$ and by Proposition \ref{a=d}, we have
\begin{align*}\sum_{n=1}^\infty\frac{d_n(\be,x)}{\be^n\la^n}&=\sum_{n=1}^{L(x)}\frac{a_n(\be,x)}{\be^n\la^n}-\frac{1}{(\be\la)^{L(x)}}+\sum_{n=1}^\infty\frac{d_n(\be,1)}{\be^n\la^n}\frac{1}{(\be\la)^{L(x)}} \\&=\sum_{n=1}^{L(x)}\frac{a_n(\be,x)}{\be^n\la^n}-\frac{1}{(\be\la)^{L(x)}}+\frac{1}{(\be\la)^{L(x)}} \\&=\sum_{n=1}^{\infty}\frac{a_n(\be,x)}{\be^n\la^n},\end{align*}which gives the desired result.

Since $\{d_n(\be,x_0)\}_{n=1}^\infty$ is the digit sequence of the quasi-greedy expansion of $x_0$, there is a sequence of positive integers $\{l(k)\}_{k=1}^\infty$ such that $l(k)<l(k+1)$ for any $k\geq1$ and $d_{l(k)}(\be,x_0)>0$ for all $k\geq1$. 
By Proposition \ref{Parry1}, we can take a number $x_k$ for $k\geq1$ whose greedy expansion is given by $x_k=\sum_{n=1}^{l(k)}d_n(\be,x_0)/\be^n$. Note that $x_k<x_{k+1}$ for $k\geq1$ and $\lim_{k\to\infty}x_k=x_0$. For sufficiently large $k$ so that $x_k\in(x_0-\delta, x_0]$, we have
\begin{align*}
\Biggl|\frac{F_\la(x_k)-F_\la(x_0)}{x_k-x_0} \Biggr|
&\geq \frac{F_\la(x_{k+1})-F_\la(x_k)}{|x_k-x_0|}-\frac{|x_{k+1}-x_0|}{|x_k-x_0|}\frac{F_\la(x_{k+1})-F_\la(x_0)}{|x_{k+1}-x_0|} \\
&\geq\frac{1-1/\be}{[\be]}\frac{(\be|\la|)^{-l(k+1)}}{(\be)^{-l(k+1)}}-C \\
&= \frac{1-1/\be}{[\be]}\frac{1}{|\la|^{l(k+1)}}-C\to\infty
\end{align*}
as $k\to\infty$, which gives the contradiction. 

\end{proof}

\begin{remark}
By using the equation 
\[\frac{a_n(\be,x)}{\be^n}=\frac{\tau^{n-1}_\be(x)}{\be^{n-1}}-\frac{\tau^{n}_\be(x)}{\be^{n}}\]
for $n\geq1$ and $x\in[0,1]$, we obtain
\begin{align*}
F_\la(x)&=\sum_{n=1}^\infty\frac{a_n(\be,x)}{\be^n\la^n}=\sum_{n=1}^\infty\frac{\tau^{n-1}_\be(x)}{\be^{n-1}\la^n}-\sum_{n=1}^\infty\frac{\tau^{n}_\be(x)}{\be^{n}\la^n} \\
&=\frac{x}{\la}+\Biggl(\frac{1}{\la}-1\Biggr)\sum_{n=1}^\infty\frac{\tau^{n}_\be(x)}{\be^{n}\la^n}
\end{align*}
for $x\in[0,1]$. We note that the function $\displaystyle{\rho_\be(x)=\sum_{n=1}^\infty\frac{\tau^{n}_\be(x)}{\be^{n}\la^n}}$ appearing in the right side of the above equation can be regarded as a version of the Takagi function, which is obtained if we replace $\be\la$ to $2$ and the map $\tau_\be$ to the tent map in the definition of $\rho_\be(x)$. The Takagi function is known as a typical example of a singular function as it is continuous but nowhere differentiable everywhere, which has been investigated in the context of fractal geometry (for more detail see \cite{Al-Ka}). 
\end{remark}

As an application of Theorem \ref{main a} (5), we show that an eigenfunctional corresponding to each non-leading eigenvalue can not be expressed by any complex measure. This is contrast to the case of the leading eigenvalue $1$ since by Theorem \ref{main a} (4) the normalized eigenfunctional of the leading eigenvalue $1$ is expressed by the Lebesgue measure.

\begin{theorem}\label{main b}
Let $\be>1$ be a non-integer and let $\la\in\C$ be a non-leading eigenvalue of the Perron-Frobenius operator $\cL_{\be}$. Then for any corresponding eigenfunctional $\nu\in\BV^*$, there is \textbf{no} complex-valued measure $m$ such that 
\[\nu(f)=\int_{0}^1 f\ dm\]
for $f\in \BV$. 
\end{theorem}

\begin{proof}
Let $\la$ be a non-leading eigenvalue of $\cL_\be$ and let $\nu_\la$ be the corresponding eigenfunctional with $\nu(\1_{[0,1]})=1$. Assume that there is a complex-valued measure $m$ such that $\nu_\la(f)=m(f)$ for any $f\in \BV$. By taking $f=\1_{[0,x]}$ for $x\in[0,1]$, we obtain
\[m([0,x])=m(\1_{[0,x]})=\nu_\la(\1_{[0,x]})=F_\la(x).\]
Since $m$ is a complex measure, the function $m([0,x])$ for $x\in[0,1]$ is of bounded variation. In particular, it is differentiable for almost every $x\in[0,1]$ with respect to the Lebesgue measure. This contradicts the fact that $F_\la(x)$ is nowhere differentiable on $[0,1]$ followed from Theorem \ref{main a} (5). 

\end{proof}

In the rest of this section, we consider a real number $\be\in(1,\infty)$ for which the corresponding Perron-Frobenius operator has at least one non-leading eigenvalue. Let $M(\be)=\max\{|\la|\ ; \la\in\C\ \text{ is a non-leading eigenvalue}\}$. Note that $1/\be< M(\be)<1$ since the modulus of each non-leading eigenvalue is in $(1/\be, 1)$ and eigenvalues are isolated in the set $\{z\in\C\ ;\ 1/\be<|z|\leq1\}$. We say that a non-leading eigenvalue $\la\in\C$ is sub-leading if its modulus is equal to $M(\be)$. 
For a function $f\in \BV$, we say that $f$ has a good decay rate of correlations if there are a positive constant $C>0$ and a positive number $\alpha$ with $1/\be\leq\alpha <M(\be)$ such that 
\[||\cL_\be^n f||_{\BV}\leq C\alpha^n\]
for $n\geq1$. Note that if a function $f\in \BV$ has a good decay rate of correlations then $\int_0^1 f\ dm=0$ and it holds
\[\Biggl|\int_{0}^1f \cdot g\circ\tau_\be^n d\mu_\be-\int_0^1f d\mu_\be\int_0^1g d\mu_\be \Biggr|\leq C|g|_\infty \alpha^n\]
for $g\in L^\infty(\mu_\be)$ and $n\geq1$ by the spectral decomposition (\ref{decomposition}) of $\cL_\be$ (see e.g. \cite{Ba}, \cite{Bo-Go}, \cite{Ho-Ke}). 
The following result gives a sufficient condition that a step function has a good decay rate of correlations.  
\begin{theorem}\label{main b2}
Let $\beta\in(1,\infty)$ be a real number such that the corresponding Perron-Frobenius operator $\cL_{\beta}$ actually has a non-leading eigenvalue. Assume that all the sub-leading eigenvalues $\{\la_i\}_{i=1}^N$ are simple. Then a step function $f:[0,1]\to\C$ of the form 
\[f(x)=\sum_{i=0}^{N+K}c_i\1_{[0,x_i]}(x)\]
for $x\in[0,1]$, where $K\geq1$, $c_i\in\C$ for $0\leq i\leq N+K$ and $x_i\in[0,1]$ for $1\leq i\leq N+K$ with $x_j\neq x_k$ if $j\neq k$, has a good decay rate of correlation if the coefficients $\{c_i\}_{i=0}^{N+K}$ is a solution of the following $(N+1)\times(N+K+1)$ matrix equation
\begin{equation}\label{matrix}
\begin{pmatrix}
x_0 & \cdots & x_{N+K}\\
F_{\la_1}(x_0) &    \cdots   & F_{\la_{1}}(x_{N+K}) \\
\vdots &   \ddots & \vdots \\
 F_{\la_N}(x_0)& \cdots   & F_{\la_N}(x_{N+K})
\end{pmatrix}
\begin{pmatrix}
c_0 \\
c_1 \\
\vdots  \\
c_{N+K}
\end{pmatrix}
=0,
\end{equation}
where $F_{\la_i}(x)$ is the function defined by
\[F_{\la_i}(x)=\sum_{n=1}^\infty\frac{a_n(\be,x)}{\be^n\la_i^n}\]
for $x\in[0,1]$ and $1\leq i\leq N$.

\end{theorem}

\begin{remark}
(1) In Section 9.2 in \cite{Co-Ec}, the authors provide an example of a piecewise linear unimodal Markov map for which isolated eigenvalues and the corresponding eigenfunctions of its Perron-Frobenius operator can be given explicitly. As an application, they also gave locally constant functions which have a good decay rate of correlation in the sense of this paper. From the above theorem we can provide additional examples of a larger class of step functions which have a good decay rate of correlations in the setting of beta-maps. 

(2) The assumption that all the sub-leading eigenvalues are simple 
seems to hold in many cases, including the case of Parry numbers less than $3$ whose beta conjugates are not units (see Corollary 3.9 in \cite{Ve}).
In Appendix, we construct a countable set of $\be$'s with $\be>3$ such that all sub-leading eigenvalues of the corresponding Perron-Frobenius operator are simple. 

\end{remark}

\begin{proof}[Proof of Theorem \ref{main b2}]
Since the dimension of the solution space of  an $(N+1)\times(N+K+1)$-matrix is at least $K$, 
there exactly exists an $(N+K+1)$-dimensional vector satisfying $(\ref{matrix})$.
We shall show a step function $f(x)=\sum_{i=0}^{N+K}c_i\1_{[0,x_i]}(x)$ whose coefficients $\{c_i\}_{i=0}^{N+K}$ satisfying the equation $(\ref{matrix})$ belongs to the kernel of the normalized eigenfunctional $\nu_{\la_i}$ of $\la_i$ for $1\leq i\leq N$ with $\nu_{\la_i}(\1_{[0,1]})=1$ and $\int_0^1 f dm=0$. By calculating the inner product of the first raw of the matrix given in $(\ref{matrix})$ and the vector $(c_0, \dots, c_{N+K})$, we have
\[\int_{0}^1 f dm=\sum_{i=0}^{N+K} c_ix_i=0.\]
Furthermore, by Proposition \ref{key2} and the equation $(\ref{matrix})$, we obtain
\[\nu_{\la_j}(f)=\nu_{\la_i}\Bigl(\sum_{i=0}^{N+K} c_i\1_{[0,x_i]}\Bigr)=\sum_{i=0}^{N+K} c_iF_{\la_j}(x_i)=0\]
for $1\leq j\leq N$.
By the spectral decomposition ($\ref{decomposition}$) of the Perron-Frobenius operator $\cL_\be$, we know that the terms for the leading eigenvalue $1$ and the subleading eigenvalues $\{\la_i\}_{i=1}^N$ in the sum of the right side of the equation ($\ref{decomposition}$)
vanish, which finishes the proof.
\end{proof}

\end{section}

\section{Appendix}
This appendix is devoted to giving three examples of countable sets of $\be$'s such that all subleading eigenvalues of $\cL_\be$ are simple.

\begin{example}
For a positive integer $n\geq3$, we give an example of $\be\in(n,n+1)$ such that it is a simple Parry number and the Perron-Frobenius operator $\cL_\be$ has only one non-leading eigenvalue on the negative axis, which is simple. We define the polynomial $p_n(x)$ by 
\[p_n(x)=x^4-nx^3-(n-1)x^2-n.\]
Since
\[p_n'(x)=x(4x^2-3nx-2(n-1)),\]
the function $p_n(x)$ for $x\in\R$ has the two local minimums at 
\[x^+:=\frac{3n+\sqrt{9n^2+8(n-1)}}{8}>\frac{3n}{2} \text{ and }x^-:=\frac{3n-\sqrt{9n^2+8(n-1)}}{8}>-1
\]
and the local maximum at $0$. Hence the function $p_n(x)$ is increasing on $(x^-,0)$ and $(x^+,+\infty)$, and is decreasing on $(-\infty, x^{-}]$ and $[0,x^+]$. Since $p_n(0)=-n<0$ and $p_n(-1)=2-n<0$, the function $p_n(x)$ has two real solutions $\be$ with $3\leq n<\be<n+1$ and $-\alpha$ with $-\alpha<-1$. In addition, $p_n(x)$ has two complex solutions $\gamma, \bar{\gamma}\in\C$ satisfying $\be(-\alpha)|\gamma|^2=-n$, which yields
\[|\gamma|^2=\frac{n}{\be\alpha}<\frac{n}{n\cdot1}=1.\]
By the fact that $\be$ satisfies
\[1=\frac{n}{\be}+\frac{n-1}{\be^2}+\frac{0}{\be^3}+\frac{n}{\be^4},\]
we have that $a_1(\be,1)=n$, $a_2(\be,2)=n-1$, $a_3(\be,1)=0$ and 
$\tau_\be^3(1)=n/\be$ by the definition of the map $\tau_\be$. 
Since $p_n(x)=x^4(1-\phi_\be(\be/x))$ for $x\in\R$ with $|x|<\be$, we have that $1-\phi_\be(z)$ has a unique zero $-\be/\al\in(-\be,-1)\subset\{z\in\C; 1<|z|<\be\}$, which yields $-\al/\be\in(-1,-1/\be)$ is a unique non-leading eigenvalue of $\cL_\be$, which is simple.   

\end{example}

\begin{example}
For a positive integer $n\geq4$, we give an example of $\be\in(n,n+1)$ such that it is a simple Parry number and
the Perron-Frobenius operator $\cL_\be$ has two complex non-leading eigenvalues, which are simple. 
We define the polynomial $q_n(x)$ by
\[q_n(x)=x^4-nx^3-(n-1)x-n.\]
Since
\[q_n'(x)=4x^3-3nx^2-(n-1)\]
and 
\[q_n''(x)=6x(2x-n)\]
the equation $q_n'(x)=0$ has only one real solution $x^+\in(n/2, \infty)$ by the fact that $q_n'(0)=-(n-1)<0$. This shows that $q_n'(x)<0$ for $x\in(-\infty, x^+)$ and $q_n'(x)>0$ for $x\in(x^+,\infty)$. Hence we have that the function $q_n(x)$ for $x\in\R$ has the local minimum at $x^+$, which yields that it is decreasing on $(-\infty, x^+)$ and increasing on $(x^+,\infty)$.
Since $q_n(-1)=n>0$, $q_n(1)=-3n+2<0$, $q_n(n)<0$ and $q_n(n+1)>0$ the equation $q_n(x)=0$ has the two real solutions $\be$ with $n<\be<n+1$ and $-\alpha$ with $0<\alpha<1$. Furthermore, it has complex solutions $\gamma, \bar{\gamma}\in\C$ satisfying $\be(-\alpha)|\gamma|^2=-n$, which implies
\[|\gamma|^2=\frac{n}{\alpha\be}>\frac{n}{(n+1)\alpha}>1,\]
since $\alpha<n/(n+1)$ by the fact that 
\begin{align*}
q_n\Bigl(-\frac{n}{n+1}\Bigr)
&>\frac{n^4}{(n+1)^3}+\frac{n(n-1)}{n+1}-n \\
&=\frac{n^4-2n(n+1)^2}{(n+1)^3}>0.
\end{align*}
Since 
\[1=\frac{n}{\be}+\frac{0}{\be^2}+\frac{n-1}{\be^3}+\frac{n}{\be^4},\]
we have that $a_1(\be,1)=n$, $a_2(\be,1)=0$, $a_3(\be,1)=n-1$ and $\tau_\be^3(1)=n/\be$ by the definition of the map $\tau_\be$. Since $q_n(x)=x^4(1-\phi_\be(\be/x))$ for $x\in\R$ with $|x|<\be$, we have that $1-\phi_\be(z)$ has two zeros $-\be/\gamma$ and $-\be/\bar{\gamma}$ with $1<|\gamma|<\be$, which yields $-\gamma/\be$ and $-\bar{\gamma}/\be$ are non-leading eigenvalues of $\cL_\be$, which are simple.  
 \begin{figure}
    \begin{tabular}{cc}
      \begin{minipage}[t]{0.45\hsize}
        \centering
        \includegraphics[width=4cm]{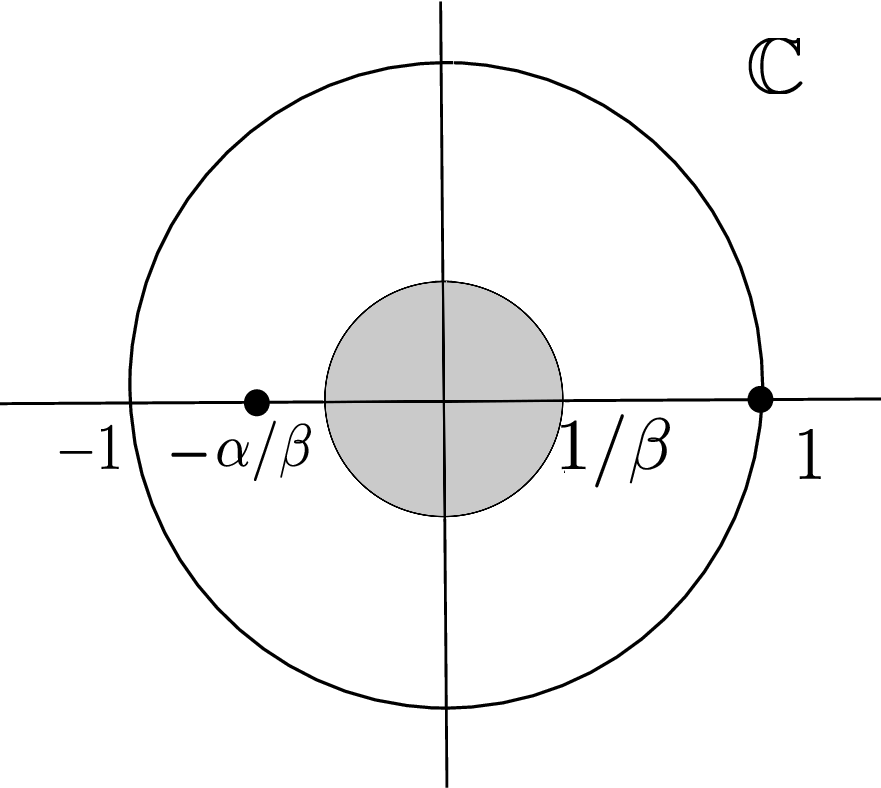}
        \caption{Figure of the spectrum of $\cL_\beta$ in Example 5.1 and 5.3.}
              \end{minipage} &
           \begin{minipage}[t]{0.45\hsize}
        \centering
        \includegraphics[width=4cm]{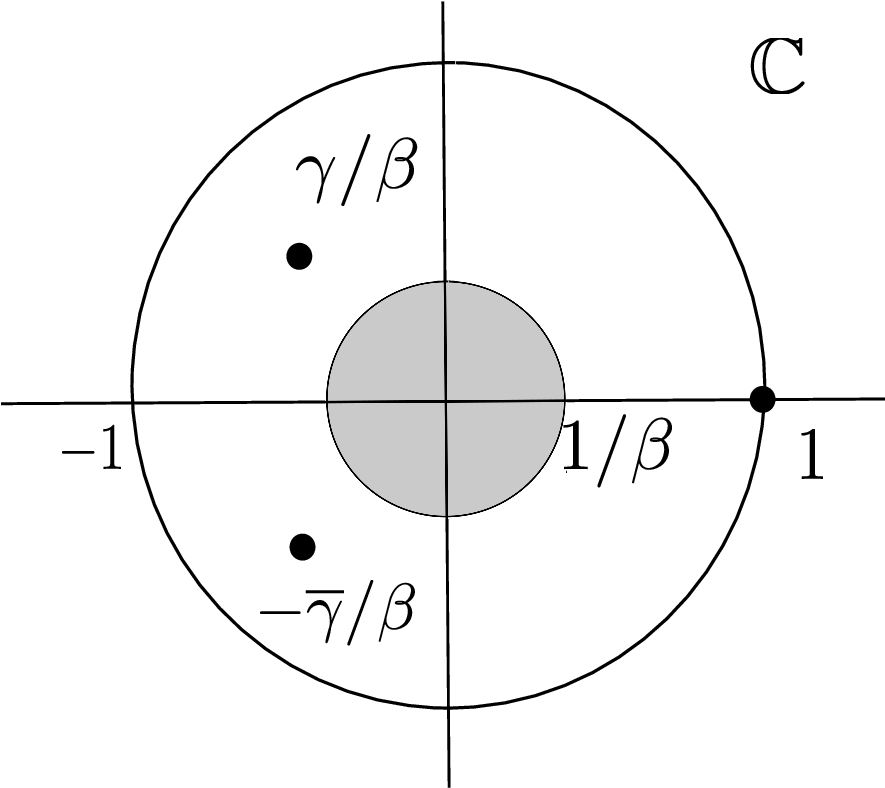}
        \caption{Figure of the spectrum of $\cL_\beta$ in Example 5.2.}
              \end{minipage}
          \end{tabular}
  \end{figure}

\end{example}

\begin{example}
For a positive integer $n\geq4$, we give an example of $\be\in(n,n+1)$ such that it is a non-simple Parry number and the corresponding Perron-Frobenius operator $\cL_\be$ has only one non-leading eigenvalue on the negative axis, which is simple. We define the polynomial $r_n(x)$ by 
\[r_n(x)=x^4-(n+1)x^3+nx-(n-1).\]
Note that
\[r_n'(x)=4x^3-3(n+1)x^2+n\]
and 
\[r_n^{''}(x)=6x(2x-(n+1)).\]
Since $r_n^{''}(x)=0$ if and only if $x=0$ or $(n+1)/2$, the function $r''_n(x)$ for $x\in\R$ has the local maximum at $0$ and the local minimum at $(n+1)/2$, the function $r'_n(x)$ for $x\in\R$ is decreasing on $(0,(n+1)/2)$ and increasing on $(-\infty, 0)\cup ((n+1)/2, \infty)$. By the fact that $r_n'(0)=n>0$ and $r'_n((n+1)/2)<r'_n(1)=1-2n<0$, the equation $r_n'(x)=0$ for $x\in\R$ has the three real solutions $x_1, x_2$ and $x_3$ with $x_1<0<x_2<1<x_3$. In fact, we can see that $x_2<1-1/n$ by the calculation:
\begin{align*}
r_n'\Bigl(1-\frac{1}{n}\Bigr)
&=4\Biggl(1-\frac{1}{n}\Biggr)^3-3(n+1)\Biggl(1-\frac{1}{n}\Biggr)^2+n \\
&<4-3(n+1)\frac{9}{16}+n=\frac{37-11n}{16}<0,
\end{align*}
since $n\geq4$. Note that the function $r_n(x)$ for $x\in\R$ has the local minimum at $x_2$ and increases on $(x_2, 1)$. Together with the fact that $0<x_2<1-1/n$, we obtain
\[r_n(x_2)=1-n(1-x_2)-x_2^3((n+1)-x_2)<1-n(1-x_2)<0,\]
which implies that the equation $r_n(x)=0$ has exactly two real solutions $\be>0$ and $-\alpha<0$.
Since $r_n(n)<0$, $r_n(n+1)>0$ and $r_n(-1)=3-n<0$, we obtain that $\be\in(n,n+1)$ and $-\alpha<-1$.
By the relation between the roots and the coefficients of the equation $r_n(x)=0$, the rest two complex solutions $\gamma, \bar{\gamma}\in\C$ satisfies
\[|\gamma|^2=\frac{n-1}{\alpha\beta}<\frac{n-1}{n}<1,\]
which yields that $|\gamma|<1$. Since
\[1=\frac{n}{\be}+\frac{n}{\be^2}+\frac{0}{\be^3}+\frac{n-1}{\be^4}+\frac{n-1}{\be^5}+\dots,\]
we have that $a_1(\be,1)=a_2(\be,1)=n$, $a_3(\be,1)=0$ and $a_{4+k}(\be)=n-1$ for any $k\geq0$ by the definition of the map $\tau_\be$. 

Since $r_n(x)=(x^4-x^3)(1-\phi_\be(\be/x))$ for $x\in\R$ with $|x|<\be$, we have that $1-\phi_\be(z)$ has a unique zero $-\be/\al\in(-\be,-1)\subset\{z\in\C; 1<|z|<\be\}$, which yields $-\al/\be\in(-1,-1/\be)$ is a unique non-leading eigenvalue of $\cL_\be$, which is simple. \end{example}

Acknowledgements. The author would like to thank Hiroki Sumi, Hiroki Takahasi and Masato Tsujii for their valuable comments. This work was supported by JSPS KAKENHI Grant Number 20K14331 and 24K16932.

\bibliographystyle{amsplain}

\end{document}